\documentclass[12pt]{amsart}
\usepackage{amsmath, amsthm, amssymb, amscd, mathptmx}
\usepackage{amsfonts}
\usepackage[all]{xy}\CompileMatrices\SelectTips{cm}{12}

\theoremstyle{plain}
\newtheorem{Thm}{\sc Theorem}[section]
\newtheorem{Theorem}[Thm]{\sc Theorem}
\newtheorem{Corollary}[Thm]{\sc Corollary}
\newtheorem{Proposition}[Thm]{\sc Proposition}
\newtheorem{Lemma}[Thm]{\sc Lemma}

\theoremstyle{definition}
\newtheorem{Definition}[Thm]{Definition}

\theoremstyle{remark}
\newtheorem{Remark}[Thm]{Remark}
\newtheorem{Remarks}[Thm]{Remarks}
\newtheorem{Example}[Thm]{Example}
\newtheorem*{Example*}{Example}
\newtheorem*{Remark*}{Remark}

\newcommand{\cA}{{\mathcal A}}

\newcommand{\cM}{{\mathcal M}}

\newcommand{\cO}{{\mathcal O}}

\newcommand{\cX}{{\mathcal X}}
\newcommand{\cY}{{\mathcal Y}}

\renewcommand{\AA}{{\mathbb A}}

\newcommand{\GG}{{\mathbb G}}

\newcommand{\PP}{{\mathbb P}}
\newcommand{\QQ}{{\mathbb Q}}

\newcommand{\VV}{{\mathbb V}}
\newcommand{\TT}{{\mathbb T}}
\newcommand{\WW}{{\mathbb W}}
\newcommand{\ZZ}{{\mathbb Z}}

\newcommand{\uM}{{\underline M}}

\newcommand{\Coh}{\mathop{\rm Coh}}

\newcommand{\cEnd}{{\mathop{\mathcal{E}nd}\,}}

\newcommand{\Gr}{\mathop{\rm Gr}}

\newcommand{\ti}{\tilde}

\newcommand{\id}{{\mathop{\rm id}}}

\newcommand{\tf}{{\mathop{\rm tf}}}

\newcommand{\Spec}{\mathop{\rm Spec \, }}

\newcommand{\QCoh}{{\mathop{{\rm QCoh}}}}

\newcommand{\Sch}{{\mathop{{\rm Sch }}}}

\newcommand{\Hom}{{\mathop{{\rm Hom}}}}
\newcommand{\cHom}{{\mathop{{\mathcal H}om}}}

\newcommand{\im}{\mathop{\rm im}}
\newcommand{\coker}{\mathop{\rm coker}}
\renewcommand{\mod}{\mathop{\rm mod}}

\newcommand{\ST}{\mathop{\rm ST}}
\newcommand{\GL}{\mathop{\rm GL}}

\begin{document}

\title{Moduli spaces of semistable modules over Lie algebroids}
\author{Adrian Langer}
\date{\today}


\maketitle

{\sc Address:}\\
Institute of Mathematics, University of Warsaw,
ul.\ Banacha 2, 02-097 Warszawa, Poland\\
e-mail: {\tt alan@mimuw.edu.pl}

\medskip

\begin{abstract} 	
	We show a few basic results about moduli spaces of semistable modules over Lie algebroids.
	The first result shows that such moduli spaces exist for relative projective morphisms of noetherian schemes, removing some earlier constraints. The second result proves general separatedness Langton type theorem for such moduli spaces. More precisely, we prove S-completness of some moduli stacks of semistable modules.	
In some special cases this result identifies closed points of the moduli space  of Gieseker semistable sheaves on a projective scheme and of the Donaldson--Uhlenbeck compactification of the moduli space of slope stable locally free sheaves on a smooth projective surface. 
	 The last result generalizes properness of Hitchin's morphism and it shows properness of so called Hodge-Hitchin morphism defined in positive characteristic on the moduli space of Gieseker semistable integrable $t$-connections in terms of the $p$-curvature morphism. This last result was proven in the curve case by de Cataldo and Zhang using completely different methods.
\end{abstract}

\section*{Introduction}

In this paper we continue the study of relative moduli spaces of semistable modules over Lie algebroids  started in \cite{La}. The aim is to show three theorems about such moduli spaces. The first result says that such moduli spaces exist in a larger  and more natural class of schemes than previously claimed. Namely, 
\cite[Theorem 1.1]{La} asserts existence of such  moduli spaces for projective families over a base of finite type over a universally Japanese ring. Here we relax these assumptions and prove existence of moduli spaces for projective families over any noetherian base (see Theorem \ref{general-moduli}).

The second result concerns \cite[Theorem 5.2]{La}, whose proof was omitted in \cite{La}. Here we prove the following much stronger version of this theorem:

\begin{Theorem}\label{main2}
	Let $R$ be a discrete valuation ring with quotient ring $K$ and residue field $k$. Let $X\to S=\Spec R$
	be a projective morphism and let us fix a relatively ample line bundle on $X/S$. Let $L$ be a smooth $\cO_S$-Lie algebroid on $X$ and let $E_1$ and $E_2$ be $L$-modules, which as $\cO_X$-modules are coherent  of relative dimension $d$ and flat over $S$. Assume that there exists an isomorphism $\varphi: (E_1)_K\to (E_2)_K$ of $L_K$-modules.
	Then we have the following implications:
	\begin{enumerate}
		\item If $(E_1)_k$ and $(E_2)_k$ are Gieseker semistable  then they are S-equivalent.
		\item  If $(E_1)_k$ and $(E_2)_k$ are Gieseker polystable  then they are isomorphic.
		\item If $(E_1)_k$ is stable and $(E_2)_k$ is Gieseker semistable then the
		$L$-modules $E_1$ and $E_2$ are isomorphic. 
	\end{enumerate}
\end{Theorem}

This theorem is a strong generalization of Langton's \cite[Theorem, p.~99]{Lt} and it implies separatedness of the moduli space of Gieseker semistable modules over a smooth Lie algebroid. Whereas this result follows from the GIT construction
of such moduli spaces, we prove a much more general result (see Theorem \ref{slope-Langton2} or below for a simple example)  that cannot be obtained in this way. 

It is well known that S-equivalent sheaves correspond to the same point in the moduli space (see, e.g., \cite[Lemma 4.1.2]{HL}). So the above theorem implies that points of the moduli space of (Gieseker) semistable modules over a smooth Lie algebroid (or, in the special case, points of the moduli space of  semistable sheaves) correspond to  S-equivalence classes of semistable modules. This was the last missing step in Faltings's non-GIT construction of the moduli space of semistable sheaves on a semistable curve (see \cite[the last paragraph on p.~509]{Fa}). A different approach to this problem for the moduli stack of semistable vector bundles over a smooth projective curve defined over a field of characteristic zero was recently developed in \cite{AHLH} (see \cite[Lemma 8.4]{AHLH}). In their language, we prove even stronger result than Theorem \ref{main2} saying that the moduli stack of Giesker semistable modules over a smooth Lie algebroid on a projective scheme is S-complete (see Theorem \ref{S-completness}). In case $k$ has characteristic zero this result follows from \cite[Proposition 3.47]{AHLH} and separatedness of the moduli space of Gieseker semistable modules but it is not so in positive characteristic. Recently, D. Greb and M. Toma pointed out to the author that  in \cite[Section 4]{GT2} they gave a direct proof of Langton's type separatedness criterion for Gieseker semistable sheaves. In their case the result does not follow from the GIT construction as they consider moduli spaces on complex projective varieties with (possibly non-ample) K\"ahler polarizations.

The original motivation to reconsider this problem was provided by Chen\-yang Xu during his talk on the ZAG seminar. Namely,  when showing proof of an analogous result for $\QQ$-Gorenstein log Fano varieties (see \cite[Theorem 1.1]{BX}), he said that there is no known direct proof of this fact for Gieseker semistable sheaves (as pointed out above, this was in fact known and proven in\cite{GT2}). 

The proof that we provide here is modelled on Gabber's proof of an analogous fact for vector bundles with integrable connections on smooth complex varieties (see \cite[Variant 2.5.2]{Ka}). The differences come mainly from the fact that coherent modules with an integrable connection on complex varieties are locally free, whereas we need to study flatness and semistability of various modules appearing in the proof. Another difference is that irreducibility of vector bundles with an integrable connection corresponds to stability in our case.

\medskip

Theorem \ref{main2} is stated for Gieseker semistability but we prove a much more general result that works also, e.g., for slope semistability. In this introduction we will formulate this result only in the simplest possible case, leaving the full generalization to Theorem \ref{slope-Langton2}.
Before formulating the result we need to slightly change the usual notion of slope semistability to allow non-torsion free sheaves. 

Let $Y$  be a smooth projective scheme defined over an algebraically closed field $k$ and let us fix an ample polarization. Let $E$ be  a coherent sheaf on $Y$ and let $T(E)$ denote the torsion part of $E$.
We say that $E$ is \emph{slope semistable} if $c_1 (T(E))=0$ and $E/T(E)$  is slope semistable  in the usual sense (we  allow $E/T(E)$ to be trivial).

We say that two slope semistable sheaves $E$ and $E'$ are \emph{strongly S-equivalent} if there exist filtrations $F_{\bullet }E$ of $E$ and $F'_{\bullet} E'$ of $E'$ such that  associated graded $\Gr ^{F}(E)$ and $ \Gr ^{F'}(E')$ are slope semistable and isomorphic to each other (see Definition \ref{semistable-filtr} and Lemma \ref{cor-Langton}). One can show that this induces an equivalence relation on slope semistable sheaves (see Corollary \ref{equivalence-relation}).

If $F_{\bullet }E$ is a filtration of $E$ then Rees construction provides a deformation of $E$ to $\Gr ^{F}(E)$. In particular, strongly S-equivalent sheaves should correspond to the same point in the ``moduli space of slope semistable sheaves''. The following result describes closed points of such ``moduli space of slope semistable sheaves''. 

\begin{Theorem}\label{main3}
	Let $R$ be a discrete valuation ring with quotient ring $K$ and residue field $k$.  
	Let $E_1$ and $E_2$ be flat families of slope semistable sheaves on $Y$ parametrized by $S=\Spec R$.
If there exists an isomorphism $\varphi: (E_1)_K\to (E_2)_K$  then we have the following implications:
	\begin{enumerate}
		\item If $(E_1)_k$ and $(E_2)_k$ are slope semistable  then they are strongly S-equivalent.
		\item  If $(E_1)_k$ and $(E_2)_k$ are slope polystable  then they are isomorphic.
		\item If $(E_1)_k$ is slope stable and torsion free and $(E_2)_k$ is slope semistable then the families $E_1$ and $E_2$ are isomorphic. 
	\end{enumerate}
\end{Theorem}

One can check that in the surface case strong S-equivalence classes correspond to closed points of the Donaldson--Uhlenbeck compactification of the moduli space (see Proposition \ref{Donaldson-Uhlenbeck}). In higher dimensions (in the characteristic zero case) there also exists an analogous construction of projective ``moduli space of slope semistable sheaves'' due to D. Greb and M. Toma (see \cite{GT}). However, their moduli space identifies many 
strong S-equivalence classes (see Example \ref{higher-dim-example}). So unlike in the surface case, closed points of their moduli spaces cannot be recovered by looking at families of (torsion free) slope semistable sheaves.

As in the previous case, the proof of  Theorem \ref{main3} shows that the moduli stack of slope semistable 
sheaves on $Y$ is S-complete. It was already known that the moduli stack of torsion free slope semistable sheaves is of finite type (see \cite{La0}) and universally closed (see \cite{Lt}). However, this stack is not S-complete (and even in characteristic zero it does not have a good moduli space, as automorphism groups of slope polystable sheaves need not be reductive). Allowing some torsion, we enlarge this stack so that it becomes S-complete and the stack of torsion free slope semistable sheaves is open in this new stack. Unfortunately, the obtained stack is no longer of finite type and it is not universally closed.

\medskip

The last result was motivated by a question of Mark A. de Cataldo asking the author about properness of the so called Hodge--Hitchin morphism. In \cite{dCZ} the authors proved such properness by finding a projective completion of the moduli space of $t$-connections in positive characteristic. Here we reprove this result in a more general setting using Langton's type theorem for restricted Lie algebroids. In fact, for a general restricted Lie algebroid the $p$-curvature defines two different proper maps but with similar proofs of properness. The main difficulty is to find an extension of a semistable module over a restricted Lie algebroid from the general fiber to the special fiber.
In general, this is not possible and it fails, e.g., for vector bundles with integrable connection over complex varieties. However, using special characteristic $p$ features one can prove an appropriate result using arguments similar to properness of the Hitchin morphism.

Then we use \cite{BMR} to show that in the case of a special Lie algebroid related to the relative tangent bundle, the two morphisms are related and we give a precise construction of the Hodge--Hitchin morphism in higher dimensions. This implies the following result (see Corollary \ref{Hodge-Hitchin-properness}).

\begin{Theorem}
Let $f : X\to S$ be a smooth morphism of noetherian schemes of characteristic $p$ and let $P$ be a Hilbert polynomial
of rank $r$ sheaves on the fibers of $f$. Then 
the Hodge--Hitchin morphism
$$M_{Hod} (X/S, P)\to \left( \bigoplus _{i=1}^r f_* \left({{S} }^i \Omega_{X/S}\right) \right) \times \AA^1$$
is proper.
\end{Theorem}

In the above theorem $M_{Hod} (X/S, P)$ denotes the moduli space of Gieseker semistable modules  with an integrable $t$-connection  and Hilbert polynomial $P$.

\medskip

The paper is organized as follows. In Section 1 we prove general existence theorem for moduli spaces of semistable modules over Lie algebroids. In Section 2 we define and study strong S-equivalence. In Section 3 we prove Theorems \ref{main2} and \ref{main3}. In Section 4 we recall some facts about the $p$-curvature of modules over restricted Lie algebroids.  Then in Section 5 we prove the results related to properness of the Hodge--Hitchin morphism in all 
dimensions.

\subsection*{Notation}

Let $\cA$ be an abelian category and let $E$ be an object of $\cA$.
All filtrations $F_{\bullet}E$ in the paper are finite and increasing. In particular, they start with the zero object and finish with $E$.

We say that a filtration $F'_{\bullet}E$ of $E$ is a \emph{refinement} of a filtration $F_{\bullet}E$ if for every $F_iE$ there exists $j$ such that $F_{j}'E=F_iE$. In this case we say that $F_{\bullet}E$ is \emph{refined} to $F'_{\bullet}E$.

\section{Moduli space of semistable  $\Lambda$-modules}

In this section we generalize the result on existence of moduli spaces of semistable modules
for projective families over a base of finite type over a universally Japanese ring to
an optimal setting of noetherian schemes. This generalization is obtained by replacing 
the use of Seshadri's results \cite{Se} on GIT quotients with a more modern technology due 
to V. Franjou and W. van der Kallen \cite{FvdK}.

\medskip

Let $f: X\to S$ be a projective morphism of noetherian schemes and let $\cO_X(1)$ be an $f$-very ample line bundle. Let $\Lambda$ be a sheaf of rings of differential operators on $X$ over $S$ (see \cite[Section 1]{La} for the definition).

Let $T$ be an $S$-scheme.  A \emph{family of $\Lambda$-modules on the fibres of} $p_T:X_T=X\times_S T\to T$ 
(or a \emph{family of $\Lambda$-modules on $X$ parametrized by $T$})
is a $\Lambda_T$-module $E$ on $X_T$, which is quasi-coherent, locally finitely presented and $T$-flat as an $\cO_{X_T}$-module. We say that $E$ is a \emph{family of Gieseker semistable $\Lambda$-modules on the fibres of} $p_T:X_T=X\times_S T\to T$ if $E$ is a family of $\Lambda$-modules on the fibres of $p_T$ 
and for every geometric point $t$ of $T$ the restriction of $E$ to the fibre $X_t$ is pure and Gieseker semistable as a $\Lambda_t$-module.

We introduce an equivalence relation $\sim $ on such families by
saying that $E\sim E'$ if and only if there exists an invertible
$\cO_T$-module $L$ such that {$E'\simeq E\otimes p_T^* L$.}

One defines the moduli functor
$$\uM ^{\Lambda}(X/S, P) : (\hbox{\rm Sch/}S) ^{o}\to \hbox{Sets} $$
from the category of schemes over $S$ to the
category of sets by
$$\uM ^{\Lambda}(X/S, P) (T)=\left\{
\aligned
&\sim\hbox{equivalence classes of families of Gieseker}\\
&\hbox{semistable $\Lambda$-modules $E$ on the fibres of } X_T\to T,\\
&\hbox{such that for every point $t\in T$ the Hilbert polynomial}\\
&\hbox{of $E$ restricted to the fiber $X_t$ is equal to $P$.}\\
\endaligned
\right\} .$$

The reason to define the moduli functor on the category of all $S$-schemes, instead of locally noetherian $S$-schemes as in \cite{HL} or \cite{Ma}, is that the moduli stack needs to be defined in that generality and one wants to relate the moduli space to the moduli stack.

We have the following theorem generalizing earlier results of
C. Simpson, the author and many others (see \cite[Theorem 1.1]{La}).

\begin{Theorem} \label{general-moduli}
	Let us fix a polynomial $P$. Then there exists a quasi-projective
	$S$-scheme $M ^{\Lambda}(X/S, P)$ of finite type over $S$ and a
	natural transformation of functors
	$$\varphi :\uM ^{\Lambda}(X/S, P)\to \Hom _S (\cdot, M ^{\Lambda}(X/S, P)),$$
	which uniformly corepresents the functor $\uM ^{\Lambda}(X/S, P)$.
	
	For every geometric point $s\in S$ the induced map $\varphi (s)$
	is a bijection. Moreover, there is an open scheme $M ^{\Lambda,
		s}(X/S, P)\subset M ^{\Lambda}(X/S, P)$ that universally
	corepresents the subfunctor of families of geometrically Gieseker
	stable $\Lambda$-modules.
\end{Theorem}

\begin{proof} 
	Here we simply sketch the changes that one needs to do in general.
	The full proof in case $\Lambda =\cO_X$ and $S$ is of finite type over 
	a universally Japanese ring is written down in the book \cite{Ma} (see \cite[Chapter 3, Theorem 9.4]{Ma} and \cite[Appendix A, Theorem 1.4]{Ma}). One of differences between our theorem and the approach presented in \cite{Ma} is that our moduli functor is defined on all $S$-schemes and not only on locally noetherian $S$-schemes as, e.g., in \cite{Ma}. This needs a slightly different definition of families in case of non locally noetherian schemes, which comes from a now standard approach taken in the construction of Quot schemes (see, e.g., \cite[Theorem 1.5.4]{Ol}).
	
	The boundedness of semistable sheaves with fixed Hilbert polynomial  $P$ 
	(see \cite{La0})
	allows one to consider an open subscheme $R$ of some Quot scheme, whose geometric points contain as quotients all semistable sheaves with fixed $P$. This $R$ comes with a group action of a certain $\GL (V)$ and one constructs the moduli scheme $M ^{\Lambda}(X/S, P)$ as a GIT quotient of $R$ by $\GL(V)$.
	
	The only place, where  one uses that  $S$  is of finite type over  a universally Japanese ring, is via Seshadri's \cite[Theorem 4]{Se} on existence of quotients of finite type (see \cite[Chapter 3, Theorem 2.9]{Ma}). Here one should point out that Seshadri's definition of a universally Japanese ring is not exactly the same as currently used as he also assumes noetherianity (see \cite[Theorem 2]{Se}).  So the notion of a universally Japanese ring used in the formulation of existence of the moduli scheme is nowadays usually called a Nagata ring.
	
	Let us recall that a  reductive group scheme over a scheme $S$ (in the sense of SGA3) is a smooth affine group scheme over $S$ with geometric fibers that are connected and reductive.
	
	We need to replace Seshadri's theorem by the following result (we quote only the most important properties)
	that follows from  \cite[Theorems 8 and 17]{vdK} (see also \cite[Theorems 3 and 12]{FvdK}):
	
	\begin{Theorem}
		Let $G$ be a reductive group scheme over a noetherian scheme $S$.  Assume that $G$ acts on a projective $S$-scheme $f: X\to S$ and there exists a $G$-linearized $f$-very ample line bundle $L$ on $X$.
		Then the following hold.
		\begin{enumerate}
			\item There is a $G$-stable open $S$-subscheme $X^{ss}(L)\subset X$, whose geometric points 
			are precisely the semistable points.  
			\item There is a $G$-invariant affine surjective morphism $\varphi: X^{ss}(L)\to Y$ of $S$-schemes  and we have $\cO_Y= \varphi _* (\cO_{X^{ss}(L)})^G$.
			\item $Y$ is projective over $S$.
		\end{enumerate}
	\end{Theorem}
	
	The proof for general $\Lambda$ is analogous to that given in \cite{Si2}.
\end{proof}

\begin{Remarks}
	
	\noindent
	\begin{enumerate}
		\item In \cite[Theorem 4.3.7]{HL} the authors add an assumption that $f: X\to S$ has geometrically connected fibers. This assumption is obsolete. 
		\item  Let us recall that $M ^{\Lambda}(X/S, P)$ {\it uniformly corepresents} $\uM ^{\Lambda}(X/S, P)$ if for every flat morphism $T\to M ^{\Lambda}(X/S, P)$ the fiber product $ T\times _{M ^{\Lambda}(X/S, P)}\uM ^{\Lambda}(X/S, P)$ is corepresented by $T$.
		The definition given in \cite[p.~512]{La} is incorrect.
		\item The need to consider moduli spaces for $X\to S$, where $S$ is of finite type over a non-Nagata ring, appeared already in some papers of de Cataldo and Zhang (see, e.g., \cite{dCZ}), where the authors consider $S$ of finite type over a discrete valuation ring. Such rings need not be Nagata rings (see \cite[Tag 032E, Example 10.162.17]{St}).
	\end{enumerate}
\end{Remarks}

\section{Strong S-equivalence}\label{strong-S-section}

In this section we introduce strong S-equivalence and study its basic properties.
We fixthe following notation.

Let $Y$ be a projective scheme over a field $k$ and let $L$ be a
$k$-Lie algebroid  on $Y$.  Let $\Coh ^L_d(Y)$ be the full subcategory of
the category of $L$-modules which are coherent as $\cO_Y$-modules
and whose objects are sheaves supported in dimension $\le d$. 

For any $d'\le d$ the subcategory $\Coh ^L_{d'-1}(Y)$ is a Serre subcategory and we can form 
the quotient category $\Coh ^L_{d,d'}(Y)=\Coh ^L_{d}(Y)/\Coh ^L_{d'-1}(Y)$. 
Let $S_{d'}$ be the class of morphisms $s: E\to F$ in $\Coh ^L_{d}(Y)$ that are isomorphisms in dimension $\le d'$, 
i.e., such that $\ker s$ and $\coker s$ are supported in dimension $<d'$. This is a multiplicative system and
$\Coh ^L_{d, d'}(Y)$ is constructed as $S_{d'}^{-1}\Coh ^L_{d}(Y)$. So the objects of $\Coh ^L_{d, d'}(Y)$
are objects of $\Coh ^L_{d}(Y)$ and morphisms in $\Coh ^L_{d, d'}(Y)$ are equivalence classes of diagrams
$E\stackrel{s}{\leftarrow} F'\stackrel{f}{\rightarrow }F$ in which $s$ is a morphism from $S_{d'}$.

\subsection{Stability and S-equivalence}

Let us fix an ample line bundle $\cO_Y(1)$ on $Y$. 
For any $E\in \Coh ^L_d(Y)$ we write the Hilbert polynomial of $E$ as
$$P(E, m)= \chi (Y, E\otimes \cO_Y(m))=\sum _{i=0}^d \alpha _i (E) \frac{m^i}{i!}.$$

Let $\QQ [t]_d$ denote the space of polynomials in $\QQ [t]$ of degree $\le d$.

For any object $E$ of $\Coh ^L_{d}(Y)$ of dimension $d$ we define its \emph{normalized Hilbert polynomial} 
$p_{d, d'} (E)$ as an element $P(E)/ \alpha_d(E)$ of $\QQ [T]_{d, d'}=\QQ [t]_d/ \QQ [t]_{d'-1}$.  
If $E$ is of dimension less than $d$ we set $p_{d, d'} (E)=0$.
This factors to
$$p_{d, d'}: {\Coh }^L_{d,d'}(Y) \to \QQ [T]_{d, d'}=\QQ [t]_d/ \QQ [t]_{d'-1}.$$ 

\begin{Definition}
	We say that $E\in \Coh ^L_d(Y)$ is \emph{stable in  $\Coh ^L_{d,d'}(Y)$} if if for 
	all proper subobjects $E'\subset E$ (in  $\Coh ^L_{d,d'}(Y)$) we have
	$$\alpha _d(E) \cdot P(E')< \alpha _d (E')\cdot P(E) \mod \QQ [t]_{d'-1}.$$
	We say that $E\in \Coh ^L_d(Y)$ is \emph{semistable in  $\Coh ^L_{d,d'}(Y)$} if it is either 
	of dimension $<d'$ or it has dimension $d$ and
	for  all proper  subobjects $E'\subset E$ we have
	$$\alpha _d(E) \cdot P(E')\le \alpha _d (E')\cdot P(E) \mod \QQ [t]_{d'-1}.$$
\end{Definition}

\medskip

\begin{Remark}
	\begin{enumerate}
		\item If  $E\in \Coh ^L_d(Y)$ is stable  in  $\Coh ^L_{d,d'}(Y)$ then it is either isomorphic to $0$ 
		in  $\Coh ^L_{d,d'}(Y)$ (so it is of dimension $<d'$) or it has dimension $d$. Otherwise, we can find a non-zero proper subobject $E'\subset E$ and $\alpha _d(E')=\alpha _d(E)=0$, contradicting the required inequality.
		\item 
		In view of the above remark, adding the assumption on the dimension of $E$ in the definition of semistability is done because we would like semistable $E$ to have a filtration, whose quotients are stable.
	\end{enumerate}
\end{Remark}

\medskip

We say that $E\in \Coh ^L_d(Y)$ is \emph{pure in $\Coh ^L_{d,d'}(Y)$}, if it has dimension $d$ and the maximal $L$-submodule $T(E)$ of $E$ of dimension $<d$ has dimension $<d'$.

Note that if  $E\in \Coh ^L_d(Y)$ of dimension $d$ is {semistable in  $\Coh ^L_{d,d'}(Y)$} then it is pure in $\Coh ^L_{d,d'}(Y)$. Indeed, we have $\alpha _d(E) \cdot P(T(E))\le 0 \mod \QQ [t]_{d'-1},$ so $P(T(E))\in \QQ [t]_{d'-1}$, which
shows that $T(E)$ has dimension $\le d'-1$. 

\medskip

If $E\in \Coh ^L_d(Y)$ is semistable in $\Coh ^L_{d,d'}(Y)$ then there exists a Jordan--H\"older filtration
$$0=E_0\subset E_1\subset ...\subset E_m=E$$
by $L$-submodules such that $E_i/E_{i-1}$ have the same normalized Hilbert polynomial $p_{d,d'}$ as $E$. The associated graded $\Gr ^{JH}(E)=\bigoplus E_{i}/E_{i-1} $ is polystable in $\Coh ^L_{d,d'}(Y)$  and its class in $\Coh ^L_{d,d'}(Y)$
independent of the choice of a Jordan--H\"older filtration of $E$.

We say that two semistable objects $E$ and $E'$ of $\Coh ^L_{d,d'}(Y)$ are \emph{S-equivalent} if their
associated graded polystable objects $\Gr ^{JH}(E)$ and $\Gr ^{JH}(E')$ are isomorphic in  $\Coh ^L_{d,d'}(Y)$.

\subsection{Basic definitions}

\begin{Definition}\label{semistable-filtr}
	Let $E\in \Coh ^L_d(Y)$. We say that a filtration
	$$0=F_0E\subset F_1E\subset ...\subset F_mE=E$$
by $L$-submodules is \emph{$d'$-semistable} if all the quotients $F_iE/F_{i-1}E$ are 
semistable in $\Coh ^L_{d,d'}(Y)$ with $p_{d,d'}(E_i/E_{i-1})$ equal to either $0$ or $p_{d,d'} (E)$.
We say that $F_{\bullet }E$ is  \emph{$d'$-stable}, if it is $d'$-semistable and  
all quotients $F_iE/F_{i-1}E$ are  stable in $\Coh ^L_{d,d'}(Y)$.
\end{Definition}

Clearly, if $E$ admits a $d'$-semistable filtration then it is semistable in $\Coh ^L_{d,d'}(Y)$
and any $d'$-semistable filtration can be refined to a $d'$-stable filtration.
A $d'$-stable filtration generalizes slightly the notion of a Jordan--H\"older filtration. In particular, a filtration $F_{\bullet} E$ is $d'$-stable if and only if the associated graded $\Gr ^{F}(E)$ is isomorphic to $\Gr ^{JH}(E)$ in  $\Coh ^L_{d,d'}(Y)$. 
The important difference is that we allow quotients to be stable in  $\Coh ^L_{d,d'}(Y)$ but with the zero
normalized Hilbert polynomial  $p_{d,d'}$.  So quotients in a $d'$-(semi)stable filtration can contain torsion (or be torsion) even if $E$ is torsion free as an $\cO_Y$-module.

Any $d'$-stable filtration can be refined to a $d'$-stable filtration whose quotients are of dimension $<d'$ or pure on $Y$ but we will also use more general $d'$-stable filtrations.

\begin{Definition}\label{strongly-S-equivalent}
We say that $E\in \Coh ^L_d(Y)$ and $E'\in \Coh ^L_d(Y)$ are \emph{strongly S-equi\-va\-lent in $\Coh ^L_{d,d'}(Y)$} if there exist $d'$-semistable filtrations $F_{\bullet }E$ and $F'_{\bullet} E'$ whose quotients 
are isomorphic up to a permutation, i.e., if both filtrations have the same length $m$ and there exists a permutation $\sigma$ of $\{1,..., m\}$ such that for all $i=1,...,m$ the $L$-modules $F_{i}E/F_{i-1}E$ and $F'_{\sigma(i)}E'/F'_{\sigma (i-1)}E'$ are isomorphic (on $Y$). In this case we write $E\simeq_{d'}E'$.
\end{Definition}

Note that if $E$ and $E'$ are strongly S-equivalent in $\Coh ^L_{d,d'}(Y)$ then the associated graded objects $\Gr ^{F}(E)$ and $\Gr ^{F'}(E')$ are isomorphic in  $\Coh ^L_{d,d'}(Y)$ and after possibly refining the filtrations they are isomorphic to $\Gr ^{JH}(E)$ in  $\Coh ^L_{d,d'}(Y)$. In particular, $E$ and $E'$ are S-equivalent in $\Coh ^L_{d,d'}(Y)$.
However, the opposite implication is false even if $L$ is a trivial Lie algebroid. For example,
Hilbert polynomials of  strongly S-equivalent modules are equal but this does not  need to be true for S-equivalent modules. Even if the Hilbert polynomials of S-equivalent $L$-modules are equal, they do not need to be strongly S-equivalent (see Section \ref{surfaces}). 

\medskip

The following lemma gives a convenient reformulation of Definition \ref{strongly-S-equivalent}.

\begin{Lemma} \label{cor-Langton}
	Let $E, E'\in \Coh ^L_{d}(Y) $ be semistable of dimension $d$ in $\Coh ^L_{d,d'}(Y)$.  Then the following conditions are equivalent:
	\begin{enumerate}
		\item $E$ and $E'$ are strongly S-equivalent in $\Coh ^L_{d,d'}(Y)$.
		\item There exist $d'$-semistable filtrations $F_{\bullet }E$ and $F'_{\bullet} E'$ such that $\Gr ^{F}(E)\simeq \Gr ^{F'}(E')$.
	\end{enumerate}
\end{Lemma}

\begin{proof}
	The implication $(1)\Rightarrow (2)$ follows immediately from the definition.  Now let us assume that $(2)$ is satisfied.
	Let us decompose  $\Gr ^{F}(E)$ and $\Gr ^{F'}(E')$ into a direct sum of irreducible $L$-modules. 
	These decompositions induce  $d'$-semistable refinements of the original filtrations.
	But by the Krull--Remak--Schmidt theorem (see \cite[Theorem 2]{At}) we can find isomorphisms between the direct factors
	of the decomposition, so quotients of the refined filtrations are isomorpic up to a permutation.  
\end{proof}

\subsection{Properties of strong S-equivalence}

\begin{Lemma}\label{torsion-equivalence}
Let $E$ be an $L$-module, coherent as an $\cO_Y$-module.
Then any two filtrations of $E$ by $L$-submodules can be refined to filtrations by $L$-submodules, whose quotients are isomorphic up to a permutation.
\end{Lemma}

\begin{proof}
Let us consider a natural ordering on the polynomials $P\in \QQ [T]$ given by the lexicographic order of their coefficients.  Assume that for any $L$-module $\ti E$ with Hilbert polynomial $P (\tilde E)<P (E)$ and for any two  filtrations of $\ti E$ we can find refinements to filtrations, whose quotients are isomorphic up to a permutation.  

Let $F_{\bullet} E$ and $G_{\bullet} E$ be two filtrations of $E$ by $L$-submodules of lengths $m$ and $m'$, respectively. It is sufficient to show that these filtrations can be refined so that quotients of the refined filtrations are isomorphic up to a permutation (note that this sort of induction  works because $P(E)\ge 0$ for any $L$-module $E$; it is a mixture of the induction on the dimension of the support and multiplicity of an $L$-module). 

The filtrations $F_{\bullet} E$ and $G_{\bullet} E$ induce the filtrations on $F_1E$ and $E/F_1E$.
By assumption these induced filtrations can be refined to filtrations, whose quotients are isomorphic up to a permutation. But these filtrations induce refinements of the filtrations $F_{\bullet} E$ and $G_{\bullet} E$,
which proves our claim.
\end{proof}

\medskip

To simplify notation we say that $E$ is \emph{$d'$-refinable} if any two $d'$-semistable filtrations of $E$ can be refined to $d'$-semistable  filtrations, whose quotients are isomorphic up to a permutation.

\begin{Proposition}\label{equivalence-lemma}
Any  $L$-module $E\in  \Coh ^L_d(Y)$ is $d'$-refinable.
\end{Proposition}

\begin{proof}
Let us consider a short exact sequence
$$0\to T(E)\to E\to E/T(E)\to 0.$$
$d'$-semistable filtrations of $E$ induce $d'$-semistable filtrations of $T(E)$ and $E/T(E)$. By
Lemma \ref{torsion-equivalence} the filtrations of $T(E)$ can be refined to filtrations by $L$-submodules, whose quotients are isomorphic up to a permutation. Such filtrations are automatically $d'$-semistable, so 
$T(E)$ is $d'$-refinable and it is sufficient to prove that $E/T(E)$ is $d'$-refinable
and then use the corresponding filtrations to obtain the required filtrations of $E$.

So in the following we can assume that $E$ is pure of dimension $d$ on $Y$. Suitably refining the filtrations we can also assume that they are $d'$-stable.
Let  $F_{\bullet} E$ and $G_{\bullet} E$ be $d'$-stable filtrations of $E$ and assume that any $L$-module $\ti E$ pure of dimension $d$ with multiplicity $\alpha _d (\tilde E)<\alpha _d (E)$ is $d'$-refinable.

Since $E$ is pure of dimension $d$, $F_1E$ is also pure of dimension $d$ on $Y$. 
Let us take the minimal $i$ such that $F_1E\subset G_iE$ and the composition $F_1E\to G_iE\to G_iE/G_{i-1}E$ is non-zero in dimension $d$ (clearly, such $i$ must exist as one can see starting from $i=m'$ and going down). Since $G_iE/G_{i-1}E$ is pure in  $\Coh ^L_{d,d'}(Y)$  and both $F_1E$ and $G_iE/G_{i-1}E$ are 
stable in $\Coh ^L_{d,d'}(Y)$ with the same normalized Hilbert polynomial $p_{d,d'}$, the map 
$F_1E\to G_iE/G_{i-1}E$ is an inclusion and an isomorphism in dimension $d'$.
So there exists a closed subset $Z\subset Y$ of dimension $\le d'-1$ such that 
the canonical map 
$$G_{i-1}E\oplus F_1E\to  G_{i}E$$
is an isomorphism on the open subset $U=Y-Z\subset Y$. Since $G_{i-1}E\oplus F_1E$ is pure of dimension $d$, this map is also injective. Therefore the composition $G_{i-1}E\to E\to E/F_1E$ is also injective.

The quotient $\tilde E= E/F_1E$ has two natural $d'$-stable filtrations induced from $E$.
The first one is defined by  $F'_{j}\ti E=F_{j+1}E/F_1E$ and
the second one by $$G'_{j}\ti E= \im ( G_{j}E\to E\to E/F_1E ).$$
Note that both filtrations are $d'$-stable. This is clear for $F'_{\bullet} \ti E$. 
Since $G_{i-1}E\to E/F_1E$ is injective
we have 
$$G'_{j}\ti E/G'_{j-1}\ti E\simeq G_jE/G_{j-1}E$$
for $j<i$. Since $F_1E\subset G_iE$, we also have 
$$G'_{j}\ti E/G'_{j-1}\ti E\simeq G_jE/G_{j-1}E$$
for $j>i$. Finally, $G'_{i}\ti E/G'_{i-1}\ti E$ is isomorphic to the cokernel of $G_{i-1}E\oplus F_1E\to  G_{i}E$, so it is either $0$ or of dimension $\le d'-1$, which proves that the filtration $G'_{\bullet} \ti E$
is $d'$-stable. Therefore we can apply the induction assumption to $\tilde E$ and then lift the corresponding filtrations to the required filtrations of $E$.
\end{proof}

\begin{Corollary}\label{equivalence-relation}
	Strong S-equivalence in $\Coh ^L_{d,d'}(Y)$ is an equivalence relation on  $\Coh ^L_{d}(Y)$.
\end{Corollary}

\begin{proof}
To simplify notation we say that filtrations satisfy condition $(*)$ if  their quotients 
are isomorphic up to a permutation (see Definition \ref{strongly-S-equivalent}).
Let us consider $E, E', E''\in  \Coh ^L_{d}(Y)$ and assume that $E\simeq_{d'}E'$
and $E'\simeq_{d'}E''$.
Then $E$ and $E'$ have the filtrations satisfying $(*)$, and $E'$ and $E''$ have the filtrations satisfying  $(*)$. 
By Proposition \ref{equivalence-lemma} the filtrations of $E'$ can be refined to filtrations satisfying 
condition $(*)$. But these refined filtrations induce filtrations of $E$ and $E''$ 
that  satisfy condition $(*)$, so $E\simeq_{d'}E''$. The remaining conditions are obvious.
\end{proof}

\medskip

Thanks to the above corollary one can talk about strong S-equivalence classes of (semistable) $L$-modules.

\begin{Lemma}\label{S-equivalence-on-sequences}
	Let us fix $P\in \QQ [T]_{d, d'}$. Let 
	$$0\to E'\to E\to E''\to 0$$
	and 
	$$0\to \ti E'\to \ti E\to \ti E''\to 0$$
	be short exact sequences of $L$-modules in $ \Coh ^L_d(Y)$ with normalized Hilbert polynomials $p_{d,d'}$ equal to either $0$ or $P$. 
	\begin{enumerate}
		\item If $E'\simeq_{d'} \ti E'$ and $E''\simeq_{d'} \ti E''$ then $E\simeq_{d'} \ti E$.
		\item If $E\simeq_{d'} \ti E$ and $E''\simeq_{d'} \ti E''$ then $E'\simeq_{d'} \ti E'$.
		\item If $E\simeq_{d'} \ti E$ and $E'\simeq_{d'} \ti E'$ then $E''\simeq_{d'} \ti E''$.
	\end{enumerate}
\end{Lemma}

\begin{proof}
The first assertion follows from the fact that a  $d'$-semistable filtration of $E$ induces  $d'$-semistable filtrations on $E'$ and $E''$ and  $d'$-semistable  refinements of these filtrations induce a  $d'$-semistable refinement of the original filtration of $E$.		

To prove the second assertion let us consider two $d'$-semistable  filtrations $F_{\bullet} E''$ and
$F_{\bullet} \ti E''$, whose quotients are isomorphic up to a permutation. 
Let $F_0 E=0$ and let $F_iE$ for $i>0$ be the preimage of $F_{i-1}E''$. This defines a $d'$-semistable
filtration  $F_{\bullet} E$ of $E$. Similarly, we can define the filtration 
$F_{\bullet} \ti E$. By Lemma \ref{equivalence-lemma} and our assumption we can find $d'$-semistable refinements of $F_{\bullet} E$ and $F_{\bullet} \ti E$ with quotients isomorphic up to a permutation.
These refinements define filtrations on $E'$ and $\ti E'$. Possibly changing the permutations we see that the quotients of these filtrations are isomorphic up to a permutation, which shows that $E'$ and $\ti E'$  are strongly S-equivalent  in $ \Coh ^L_{d,d'}(Y)$. 
The last assertion can be proven similarly to the second one.
\end{proof}

\subsection{Slope semistability on surfaces} \label{surfaces}

To better understand strong S-equivalence classes let us consider the surface case.

Let $Y$ be a smooth projective surface, $d=2$, $d'=1$ and $L=\cO_Y$ is the trivial Lie algebroid. In this case a coherent sheaf $E$ is semistable in  $\Coh ^L_{2,1}(Y)$ if and only if $E$ is slope semistable
(of dimension $2$ or $0$).

\begin{Lemma}\label{independence}
	Let $E$ be a slope semistable sheaf of dimension $2$ or $0$ on $X$. Then 
for any $1$-stable filtration $F_{\bullet}E$ the sheaf 	$({\Gr} ^{F}(E)) ^{**}$
and the function  $l_E: Y\to \ZZ_{\ge 0}$  given by 
$$l_E(y)={\mathrm {length}} \, ( \ker  ({\Gr} ^{F}(E)\to  ({\Gr} ^{F}(E))^{**}) _y + {\mathrm {length}} \, ( \coker  ({\Gr} ^{F}(E)\to  ({\Gr} ^{F}(E))^{**}) _y$$	
do not depend on the choice of the filtration.
\end{Lemma}

\begin{proof}
If $E$ has dimension $0$ the assertion is clear as the reflexivization is trivial and $l_E(y)={\mathrm {length}} \, E_y$.  So in the following we can assume that $E$ has dimension $2$. 

In this proof we write $l_E^F$ for the function $l_E$ defined by the filtration $F_{\bullet} E$.
By Lemma \ref{equivalence-lemma} any two $1$-stable filtrations of $E$ can be refined to $1$-stable filtrations, whose quotients are isomorphic up to a permutation. So it is sufficient to prove that  if $F'_{\bullet}E$ is a refinement of a $1$-stable filtration $F_{\bullet}E$ then 
	$({\Gr} ^{F}(E)) ^{**}\simeq 	({\Gr} ^{F'}(E)) ^{**}$ and $l_E^F=l_E^{F'}$.
Then passing to the quotients, we can reduce to the situation when $F_{\bullet} E$ has length $1$, i.e.,
$E$ is  slope stable of dimension $2$ or $0$ on $X$. In the second case we already know the assertion, so we can assume that $E$ has dimension $2$. 
Let us consider a short exact sequence
$$0\to E'\to E\to E''\to 0$$
in which one of the sheaves $E'$ and $E''$ is $0$-dimensional and the other one is $2$-dimensional and slope stable. 

Let us first assume that $E'$ is $0$-dimensional. Then  $E^{**}\simeq (E'')^{**}$,
$$\coker  (E\to E^{**})\simeq \coker  (E'\to  (E'')^{**})$$
and we have  a short exact sequence
$$0\to E'\to T(E)\to T(E'')\to 0.$$
Similarly, if $E''$ is $0$-dimensional we have $E^{**}\simeq (E')^{**}$,
$$\coker  (E'\to  (E')^{**}) \simeq \coker  (E\to E^{**}) $$
and we have  a short exact sequence
$$0\to T(E')\to T(E)\to E''\to 0.$$
So we see that  $({\Gr} ^{F'}(E)) ^{**}\simeq E^{**}$ and $l_E^F=l_E^{F'}$ follows by induction on the length of the filtration $F'_{\bullet }E$.
\end{proof}

\begin{Corollary}\label{one-implication}
If $E$ and $E'$ are semistable and strongly  S-equivalent  in  $\Coh ^L_{2,1}(Y)$ then $({\Gr} ^{JH}(E) )^{**} \simeq ({\Gr} ^{JH}(E') )^{**}$ and $l_E=l_{E'}.$
\end{Corollary}

\begin{proof}
Assume that $E\simeq_{1}E'$ and let $F_{\bullet }E$ and $F'_{\bullet} E'$ be $1$-stable filtrations whose quotients are isomorphic up to a permutation.
Since ${\Gr} ^{F}(E)\simeq {\Gr} ^{F'}(E)$, we have $({\Gr} ^{F}(E)) ^{**}\simeq({\Gr} ^{F'}(E')) ^{**}$ and $l_E=l_{E'}$. So the corollary follows from the fact that by Lemma  \ref{independence} we have $({\Gr} ^{JH}(E)) ^{**}\simeq({\Gr} ^{F}(E)) ^{**}$ and   $({\Gr} ^{JH}(E')) ^{**}\simeq({\Gr} ^{F'}(E')) ^{**}$.
\end{proof}

\medskip

If $E$  is slope semistable and torsion free then we can find a (slope) Jordan--H\"older filtration $E^{\tf}_\bullet $ of $E$ such that the associated graded $\Gr ^{\tf}(E)$ is also torsion free.  
Then for any $y\in Y$ we have
$$l_E(y)={\mathrm {length}} \, (({\Gr} ^{\tf}(E))^{**}/ {\Gr} ^{\tf}(E)) _y ,$$
so our function $l_E$ agrees with the one from \cite[Definition 8.2.10]{HL}.

Let  $M^{\mu \mathrm {ss}}(r, \Lambda, c_2)$ be the moduli space of torsion free slope semistable sheaves $E$ of rank $r$
with $\det E\simeq \Lambda$ and $c_2(E)=c_2$ (see \cite[8.2]{HL}).
The following result shows that closed points of $M^{\mu \mathrm {ss}}(r, \Lambda, c_2)$ correspond to strong  S-equivalence classes (see \cite[Theorem 8.2.11]{HL}).

\begin{Lemma} \label{Donaldson-Uhlenbeck}
Let $E$ and $E'$ be slope semistable torsion free sheaves on $X$. Then $E$ and $E'$ are strongly 
S-equivalent  in  $\Coh ^L_{2,1}(Y)$ if and only if $({\Gr} ^{\tf}(E) )^{**} \simeq ({\Gr} ^{\tf}(E') )^{**}$ and $l_E=l_{E'}.$
\end{Lemma}

\begin{proof}
One implicaton follows from Corollary \ref{one-implication}.	
To prove the other implication let us assume that  $({\Gr} ^{\tf}(E)) ^{**}\simeq ({\Gr} ^{\tf}(E') )^{**}$
and $l_E=l_{E'}$. Since $E\simeq _1{\Gr} ^{\tf}(E)$ and  $E'\simeq _1{\Gr} ^{\tf}(E')$, by Corollary \ref{equivalence-relation} it is sufficient to prove that ${\Gr} ^{\tf}(E) \simeq _1{\Gr} ^{\tf}(E')$.
So we can assume that $E$ and $E'$ are slope polystable. 
Since $E^{**}/E$ and $(E')^{**}/E'$ are $0$-dimensional sheaves of the same length at every point of $Y$, we can find the filtrations of $E^{**}/E$ and $(E')^{**}/E'$ whose quotients are isomorphic. Therefore
$E^{**}/E\simeq_1(E')^{**}/E'$. Since by assumption $E^{**}\simeq _1(E')^{**}$,
Lemma \ref{S-equivalence-on-sequences} implies that $E\simeq_1 E'$.
\end{proof}

\subsection{Slope semistability in higher dimensions}

Let $Y$ be a smooth projective variety of dimension $d$. Let us consider $d'=d-1$ and the trivial Lie algebroid  $L=\cO_Y$.
Then a coherent sheaf $E$ is semistable in   $\Coh ^L_{d,d'}(Y)$ if and only if $E$ is slope semistable
(of dimension $d$ or $\le d-2$).  

\begin{Example} \label{higher-dim-example}
Let $E$ be a slope stable vector bundle on $Y$ of dimension $\ge 3$. Let us consider the family of slope stable torsion free sheaves $\{E_y\} _{y\in Y}$ defined by $E_y=\ker (E\to E\otimes k(y))$. This can be seen to be a flat family parametrized by $Y$. Note that if $y_1$, $y_2$ are distinct $k$-points of $Y$ then $E_{y_1}$ and $E_{y_2}$ are not strongly S-equivalent in $\Coh ^L_{d,d'}(Y)$. This follows from Lemma \ref{S-equivalence-on-sequences}
and the fact that $E\otimes k(y_1) $ and $E\otimes k(y_2)$ are not strongly S-equivalent in $\Coh ^L_{d,d'}(Y)$ as they have different supports.

Note however that by \cite[Lemma 5.7]{GT} all  $E_y$ for $y\in Y(k)$ correspond to the same point in the moduli space of slope semistable sheaves constructed by D. Greb and M. Toma.
\end{Example}

\section{``Separatedness of the moduli space of semistable modules''}

In this section we fix the following notation.

Let $R$ be a discrete valuation ring with maximal ideal $m$
generated by $\pi \in R$. Let $K$ be the quotient field of $R$ and
let us assume that the residue field $k=R/m$ is algebraically
closed.

Let $X\to S=\Spec R$ be a projective morphism and let $L$
be a smooth $\cO_S$-Lie algebroid on $X$. Let us fix a 
relatively ample line bundle $\cO_X (1)$ on $X/S$. In the following 
stability of sheaves on the fibers of $X\to S$ is considered with respect to 
this fixed polarization.

\subsection{Flatness lemma}

In the proof of the main theorem of this section we need the following lemma.

\begin{Lemma}\label{hom-flatness}
	Let  $E_1$ and $E_2$ be coherent $\cO_X$-modules. If $E_2$ is flat over $S$ then 
	\begin{enumerate}
		\item the $R$-module $\Hom _{\cO_X} (E_1, E_2)$ is flat,
		\item the sheaf $\cHom _{\cO_X} (E_1, E_2)$ is flat over $S$,
		\item we have a canonical isomorphism
		$$\Hom _{\cO_X} (E_1, E_2)\otimes _RK\mathop{\longrightarrow}^{\simeq} \Hom _{\cO_{X_K}} ((E_1)_K, (E_2)_K) .$$		
	\end{enumerate}
\end{Lemma}

\begin{proof}
	Since $E_2$ is $R$-flat, the canonical map $E_2\to (E_2)_K$ is an inclusion. 
	If $\varphi: E_1\to E_2$ is a non-zero $\cO_X$-linear map and $\pi \varphi =0$ then $\varphi \pi =0$,
	so $\varphi$ factors through $E_1/\pi E_1\to E_2$. But $E_1/\pi E_1$ is a torsion $R$-module and
	$(E_2)_K$ is a torsion free $R$-module (since it is a $K$-vector space and $K$ is a torsion free $R$-module), so $E_1/\pi E_1\to E_2\subset (E_2)_K$ is the zero map and hence $\varphi =0$. It follows that $\Hom _{\cO_X} (E_1, E_2)$ is a torsion free $R$-module. So it is also a free $R$-module (and in particular $R$-flat). Since $E_1$ and $E_2$ are $\cO_X$-coherent, the canonical homomorphism
	$\cHom _{\cO_X} (E_1, E_2) _x\to \Hom _{\cO_{X,x}} ((E_1)_x, (E_2)_x)$ is an isomorphism for every point $x\in X$. Then a local version of the same argument as above shows that $\cHom _{\cO_X} (E_1, E_2)$ is flat over $S$. 
	Since cohomology commutes with  flat base change, we have an isomorphism
	$$\Hom _{\cO_X} (E_1, E_2)\otimes _RK\mathop{\longrightarrow}^{\simeq} H^0(X_K, (\cHom _{\cO_X} (E_1, E_2))_K).$$
	Since $X_K\subset X$ is open we have  $(\cHom _{\cO_X} (E_1, E_2))_K 
	\simeq \cHom _{\cO_{X_K}} ((E_1)_K, (E_2)_K)$ and hence
	$$\Hom _{\cO_X} (E_1, E_2)\otimes _RK\mathop{\longrightarrow}^{\simeq} \Hom _{\cO_{X_K}} ((E_1)_K, (E_2)_K) .$$
\end{proof}

\subsection{Langton type theorem for separatedness}

The following theorem is a far reaching generalization of  \cite[Theorem 5.2]{La} and \cite[Theorem 5.4]{La} (unfortunately, the proof of the first part of \cite[Theorem 5.2]{La} was omitted). The second part of the proof is similar to Gabber's proof of  \cite[Variant 2.5.2]{Ka} but we need to study semistability of various sheaves appearing in the proof.

\begin{Theorem}\label{slope-Langton2}
	Let $E_1$ and $E_2$ be $R$-flat $\cO_X$-coherent $L$-modules of relative dimension $d$.
	Assume that there exists an isomorphism $\varphi: (E_1)_K\to (E_2)_K$ of $L_K$-modules.
	Then we have the following implications:
	\begin{enumerate}
		\item If $(E_1)_k$ and $(E_2)_k$ are semistable in $\Coh ^L_{d,d'}(X_k)$ then they are strongly S-equivalent in $\Coh ^L_{d,d'}(X_k)$.
		\item  If $(E_1)_k$ and $(E_2)_k$ are polystable in $\Coh ^L_{d,d'}(X_k)$ then they are isomorphic in $\Coh ^L_{d,d'}(X_k)$.
		\item If $(E_1)_k$ is stable and $(E_2)_k$ is semistable in $\Coh ^L_{d,d'}(X_k)$ and $(E_1)_k$ is pure then $\pi^n\varphi$ extends to an isomorphism of $L$-modules $E_1\to E_2$  for some integer $n$. 
	\end{enumerate}
\end{Theorem}

\begin{proof}
	By Lemma \ref{hom-flatness} $\Hom _{\cO_X} (E_1, E_2)$ is a free $R$-module and
$$\Hom _{\cO_X} (E_1, E_2)\otimes _RK\mathop{\longrightarrow}^{\simeq} \Hom _{\cO_{X_K}} ((E_1)_K, (E_2)_K) .$$
So if we treat $E_i$ as an $\cO_X$-submodule of $(E_i) _K$ then $\pi^n \varphi (E_1)\subset E_2$ for some integer $n$. Note that $\varphi'=\pi ^n\varphi : E_1\to E_2$ is a homomorphism of $L$-modules. More precisely, giving an $L$-module structure on $E_i$ is equivalent to an integrable $d_{\Omega _L}$-connection $\nabla_i: E_i\to E_i\otimes \Omega_L$. Under this identification $\varphi'$ is a homomorphism of $L$-modules
 if and only if $\alpha= (\varphi'\otimes \id)\circ \nabla _1-\nabla_2\circ \varphi '$ is the zero map.
 But this map is $\cO_X$-linear and $\alpha _K=0$, since $\varphi$ is a homomorphism of $L_K$-modules.
Since by Lemma \ref{hom-flatness} $\Hom _{\cO_X} (E_1, E_2\otimes \Omega _L)$ is a free $R$-module
and 
$$\Hom _{\cO_X} (E_1, E_2\otimes \Omega _L)\otimes _RK\mathop{\longrightarrow}^{\simeq} \Hom _{\cO_{X_K}} ((E_1)_K, (E_2)_K\otimes \Omega _{L_K}), $$
we have $\alpha =0 $ as required. 

Note that if we choose $n$ so that  $\pi^n \varphi (E_1)\subset E_2$ but $\pi^{n-1} \varphi (E_1)$ is not contained in $E_2$ then $\varphi'_k: (E_1)_k\to (E_2)_k$ is a non-zero map of sheaves with the same Hilbert polynomial. In particular, if $(E_1)_k$ is stable and $(E_2)_k$ is semistable in $\Coh ^L_{d,d'}(X_k)$ and $(E_1)_k$ is pure, then  $\varphi'_k$ is an isomorphism of $L_k$-modules, which implies that $\varphi '$ is an isomorphism of $L$-modules. This gives the last part of the theorem.
 	
Now let us set $\psi=\varphi^{-1}:(E_2)_K\to (E_1)_K$. Then as above for some integer $m$ we get a homomorphism of $L$-modules $\psi'=\pi ^m\psi : E_2\to E_1$. So setting $E_2'=\psi (E_2)$
we have inclusions of $L$-modules $\pi^m E_2'\subset E_1$ and $\pi ^nE_1\subset E_2'$.  
For every $(a,b)\in \ZZ^2$ let us set 
$$E(a,b):= \im (E_1\oplus E_2' \mathop{\longrightarrow}^{\pi^a\oplus \pi ^b} (E_1)_K ). $$
Each $E(a,b)$ is an $L$-submodule of $(E_1)_K$, which as an $\cO_X$-module is flat over $S$.
Moreover, we have $E(a,b)_K= (E_1)_K$.
We claim that $E(a, b)_k$ is semistable in $\Coh ^L_{d,d'}(X_k)$.
Assume it is not semistable and let $E(a, b)_k\twoheadrightarrow F$ be the minimal destabilizing quotient.
Note that $E(a, b)_k$, $(E_1)_k$ and $(E_2')_k$ are special fibers of $S$-flat families with the same general fiber $(E_1)_K$. Therefore their Hilbert polynomials, and hence also reduced Hilbert polynomials, coincide.
Since 
$$p((E_1)_k\oplus (E_2')_k)= p (E(a, b)_k)> p_{\min} (E(a,b)_k)=p (F) \mod \QQ [t]_{d'-1}$$
and $(E_1)_k\oplus (E_2')_k$ is semistable in $\Coh ^L_{d,d'}(X_k)$, every map of $L_k$-modules $(E_1)_k\oplus (E_2')_k\to F$ is zero in $\Coh ^L_{d,d'}(X_k)$. But we have a surjection $(E_1)_k\oplus (E_2')_k\to E(a, b)_k\to F$, a contradiction.

On $X_k$ we have  short exact sequences
 	$$0\to E(a+1,b)/\pi E(a,b) \to E(a,b)/\pi E(a,b) \to E(a,b)/ E(a+1,b)\to 0$$
and 
$$0\to  \pi E(a,b)/\pi E(a+1,b) \to E(a+1,b)/\pi E(a+1,b)\to E(a+1,b)/ \pi E(a,b)\to 0.$$	
Since $E(a,b)_k= E(a,b)/\pi E(a,b)$ and $E(a,b)_k$ is pure in  $\Coh ^L_{d,d'}(X_k)$, $E(a+1,b)/\pi E(a,b)$
is also pure in $\Coh ^L_{d,d'}(X_k)$. Similarly, the second sequence implies that $E(a,b)/E(a+1,b)\simeq \pi E(a,b)/\pi E(a+1,b)$ is pure in $\Coh ^L_{d,d'}(X_k)$.
If either $E(a+1,b)/\pi E(a,b)$ or $E(a,b)/E(a+1,b)$ has dimension $\le d'-1$, then the other sheaf is semistable in $\Coh ^L_{d,d'}(X_k)$ and it is clear that $E(a,b)_k$ and $E(a+1,b)_k$ are strongly S-equivalent in $\Coh ^L_{d,d'}(X_k)$.
So we can assume that both $E(a+1,b)/\pi E(a,b)$ and $E(a,b)/E(a+1,b)$ have dimension $d$.

Then the first short exact sequence shows that 
$$p_{\max} (E(a+1,b)/ \pi E(a,b))\le p (E(a, b)_k)\le p_{\min} (E(a,b)/ E(a+1,b)) \mod \QQ [t]_{d'-1}.$$
 Similarly, the second short exact sequence shows that 
$$ p_{\max} (E(a,b)/ E(a+1,b)) \le p (E(a+1,b)_k)\le p_{\min} (E(a+1,b)/ \pi E(a,b)) \mod \QQ [t]_{d'-1}.$$
 	
But $p (E(a, b)_k)=p (E(a+1,b)_k)$ so $E(a+1,b)/\pi E(a,b)$ and $ E(a,b)/ E(a+1,b)$ are semistable in $\Coh ^L_{d,d'}(X_k)$ with the same normalized Hilbert polynomial in $ \QQ [T]_{d, d'}$.
This shows that  $E(a,b)_k$ and  $E(a+1,b)_k$ are strongly S-equivalent in $\Coh ^L_{d,d'}(X_k)$.
Similarly,  $E(a,b)_k$ and  $E(a,b+1)_k$ are strongly S-equivalent in $\Coh ^L_{d,d'}(X_k)$. 
Since $E_1=E(0,m)$ and $E_2'= E(n,0)$, Corollary \ref{equivalence-relation} implies the first part of the theorem. The second part follows immediately from the first one. 
\end{proof}

\subsection{S-completness}

Let us recall some definitions from \cite[Section 3.5]{AHLH}.

If $R$ is a discrete valuation ring with a uniformizer $\pi$ then one can consider the following quotient stack
$$ \overline{\ST}_R: =[\Spec (R[s,t]/(st-\pi)) /\GG_m],$$
where $s$ and $t$ have $\GG_m$-weights $1$ and $-1$. 

\begin{Definition}
	We say that a morphism $f: \cX \to \cY$ of locally noetherian algebraic stacks is \emph{S-complete} if for any DVR $R$ and any commutative diagram 
	$$\xymatrix{
		\overline{\ST}_R \setminus 0\ar[d]\ar[r]&\cX \ar[d] \\
		\overline{\ST}_R \ar@{-->}[ru]\ar[r]&\cY.\\
	}$$
	of solid arrows, there exists a unique dashed arrow filling in the diagram.
\end{Definition}

\medskip

Let $Y$ be a projective  scheme  over an algebraically closed field $k$ and let $\cO_Y(1)$ be an ample line bundle
(one can also consider the general relative situation as in Section \ref{moduli-section} but we state the results in the simplest possible case to simplify notation). Let us also fix a smooth $k$-Lie algebroid on $Y$. 
Then we consider the moduli stack $\cM ^{L, \rm{ss}}_{d,d'}(X)$ of objects of  $\Coh ^L_{d}(Y)$ that are  semistable in $\Coh ^L_{d,d'}(Y)$.

\begin{Theorem}\label{S-completness}
The moduli stack $\cM^{L, \rm{ss}}_{d,d'}(Y)$ is S-complete over $\Spec k$.
\end{Theorem}

\begin{proof}
Let us fix a commutative diagram 
	$$\xymatrix{
	\overline{\ST}_R \setminus 0\ar[d]\ar[r]& \cM^{L, \rm{ss}}_{d,d'}(Y)\ar[d] \\
	\overline{\ST}_R \ar[r]&\Spec k \\
}$$
and let us consider  $X=Y\times _k \Spec R \to S=\Spec R$. 
Giving a morphism $\overline{\ST}_R \setminus 0\to \cM^{L, \rm{ss}}_{d,d'}(Y)$ is equivalent to giving two $R$-flat $\cO_X$-coherent $L$-modules  $E_1$ and $E_2$ such that $(E_1)_k$ and $(E_2)_k$ are semistable in $\Coh ^L_{d,d'}(X_k)$ together with
an isomorphism $\varphi: (E_1)_K\to (E_2)_K$ of $L_K$-modules.

Let us use the notation from the proof of Theorem \ref{slope-Langton2} and let us set 
$$F_j=\left\{ \begin{array}{cl}
E(-j, 0) & \hbox{for }  j\le 0,\\
E(0, j) & \hbox{for } 1\le j.\\
\end{array} \right.$$
Let us now consider the diagram of maps
$$\xymatrix{
\cdots\ar@/^1pc/[r]^{1}& F_{-2}\ar@/^1pc/[r]^{1}\ar@/^1pc/[l]^{\pi} & F_{-1}\ar@/^1pc/[r]^{1}\ar@/^1pc/[l]^{\pi}&F_0\ar@/^1pc/[r]^{\pi}\ar@/^1pc/[l]^{\pi} &F_1\ar@/^1pc/[r]^{\pi}\ar@/^1pc/[l]^{1} &F_2\ar@/^1pc/[r]^{\pi }\ar@/^1pc/[l]^{1} &\cdots \ar@/^1pc/[l]^{1}\\
}$$
By assumption $\pi ^{n+m}E_1 \subset \pi^m E_2'\subset E_1$, so $n+m\ge 0$. Replacing $E_1$ with $E_2$ if necessary, we can therefore assume that $n\ge 0$. The proof of Theorem  \ref{slope-Langton2} shows that we have $F_j=E_2'$ for $j\le -n$  and $F_j=E_1$ for $j\ge n+m$. By \cite[Remark 3.36]{AHLH} and the proof of  Theorem  \ref{slope-Langton2},
this gives the required map $\overline{\ST}_R \to \cM^{L, \rm{ss}}_{d,d'}(Y)$.
\end{proof}

\section{Modules over Lie algebroids in positive characteristic}\label{modules-char-p}

Let $f: X\to S$ be a  morphism of noetherian schemes, where $S$ is a scheme of characteristic $p>0$.
Let $L$ be a smooth restricted $\cO_S$-Lie algebroid on $X$, i.e., a locally free $\cO_X$-module $L$ equipped with a restricted $\cO_S$-Lie algebra structure (i.e., an $\cO_S$-Lie algebra structure with the $p$-th power operation) and an anchor map $L\to T_{X/S}$ compatible with $p$-th power map (see \cite[Definitions 2.1 and 4.2]{La}). By $\Lambda _L$ we denote the universal enveloping algebra of the Lie algebroid $L$ (see \cite[p.~515]{La}).

\subsection{$p$-curvature}

Let $F_X:X\to X$ denote the absolute Frobenius morphism.
Let us recall that for an $L$-module $M=(E, \nabla: L\to \cEnd _{\cO_S} E)$ we can define its $p$-curvature
$\psi (\nabla): L\to \cEnd _{\cO_S} E$ by sending $x\in L$ to $(\nabla (x))^{p}-\nabla (x^{[p]})$. In fact, this gives rise 
to a map $\psi (\nabla)  : F_X^*L\to \cEnd _{\cO_X} E$, that we also call the $p$-curvature of $M$  (see \cite[4.4]{La}).
In this case, $(E, \psi(\nabla))$ defines an $F_X^*L$-module, where $F_X^*L$ has a trivial $\cO_S$-Lie algebroid structure
(equivalently, we get an $F_X^*L$-coHiggs sheaf  $(E, E\to E\otimes _{\cO_X} F_X^*\Omega_L)$). We will reinterpret this sheaf in terms of sheaves
on the total space  $\VV( F^*_XL)$ of $F_X^*L$ as follows.

Let us recall the following corollary of \cite[Lemma 4.5]{La}, generalizing an earlier known result for the ring of differential operators:

\begin{Lemma} \label{BMR-lemma}
	The map
	$\imath : F^*_XL\to \Lambda_L$ sending $x\otimes 1\in  F^*_XL= F^{-1}_XL\otimes_{F_X^{-1}\cO_X}\cO_X$
	for $x\in L$ to $\imath(x\otimes 1):=x^p-x^{[p]}\in \Lambda_L $ is $\cO_X$-linear and its
	image is contained in the centralizer $Z_{\Lambda_L}  (\cO_X)$ of $\cO_X$  in $\Lambda_L$. Moreover, $\imath$ extends to an inclusion of
	the symmetric algebra	$S^{\bullet}(F^*_XL)$ into $Z_{\Lambda_L}  (\cO_X)$.
\end{Lemma}

The above lemma shows that there exists a sheaf of (usually non-commutative) rings ${\ti\Lambda}_{L}$
with an injective homomorphism of sheaves of rings  $\cO_{\VV ( F^*_XL)}\to {\ti\Lambda}_{L}$  such that
$$\pi _* {\ti\Lambda}_{L}={\Lambda}_L ,$$
where $\pi : \VV(F_X^*L)\to X$ denotes the canonical projection. By \cite[Lemma 4.6]{La} 
${\ti \Lambda}_L$ is  locally free of finite rank both as a left and a right $\cO_{\VV (F_X^*L)}$-module.

In the following ${\mathrm \QCoh \, } (Y, \cA)$ denotes the category of (left) $\cA$-modules, which are quasicoherent as $\cO_Y$-modules.

\begin{Lemma}\label{equivalence}
	We have an equivalence of categories 
	$${\mathrm \QCoh \, } (X, L) \simeq {\mathrm \QCoh \, } (\VV (F_X^*L), {\ti\Lambda}_{L})$$
	such that if  $M$ is an $L$-module and $\ti M$ is the corresponding $\ti \Lambda _{L}$-module then 
	$$\pi _* \ti M=M.$$
	Moreover, $M$ is coherent as an $\cO_X$-module if and only if $\ti M$ is coherent as an $\cO_{\VV (F_X^*L)}$-module
	and its support is proper over $X$.
\end{Lemma}

\begin{proof}
Since $\pi $ is affine, we have the following equivalences of categories:
$${\mathrm \QCoh \, } (X, L) \simeq {\mathrm \QCoh \, } (X, \Lambda_ L)
\simeq {\mathrm \QCoh \, } (\VV ( F^*_XL), {\ti\Lambda}_{L}).$$
If $M$ is coherent then $\ti M$ is coherent.  Let us fix a relative compactification $Y$ of $\VV(F_X^*L)$, e.g., 
$Y=\PP (F_X^*L\oplus \cO_X)\to X$.  The support of $\ti M$ is quasi-finite over $X$, so it does not intersect
the divisor at infinity $D=Y- \VV(F_X^*L)$. Since ${\mathrm {Supp}}\,  \ti M$ is closed in $\VV(F_X^*L)$, it is also closed in $Y$ and hence it is proper over $X$. On the other hand, if  $\ti M$ is coherent and the support of $\ti M$ is proper over $X$ then $M=\pi _* \ti M$ is coherent.
\end{proof}

Let $M=(E, \nabla: L\to \cEnd _{\cO_S} E)$ be an  $L$-module, quasi-coherent as an $\cO_{X}$-module and let 
$\ti M$ be a  $\ti \Lambda _{L}$-module corresponding to $M$. Let $\ti N$ denote $\ti M$ considered as 
an $\cO_{\VV (F_X^*L)}$-module. Using the standard equivalence 
$${\mathrm \QCoh \, } (X, F_X^*L) \simeq {\mathrm \QCoh \, } (\VV ( F^*_XL), \cO_{\VV( F^*_XL)}),$$
we get an $F_X^*L$-module $N=\pi _*\ti N$. Lemma \ref{BMR-lemma} shows that this module is equal to $(E, \psi (\nabla))$, which gives another interpretation of the $p$-curvature.

\subsection{Modules on the Frobenius twist} \label{F-twist}

Now let us assume that $X/S$ is smooth of relative dimension $d$. Let $F_{X/S}: X\to X'$ denote the relative Frobenius morphism over $S$ and let
$L'$ be the pull back of  $L$ via $X'\to X$.  \cite[Lemma 4.5]{La} shows that $\imath :
S^{\bullet}(F^*_XL)=F^*_{X/S}S^{\bullet}(L')\to \Lambda_L$ induces a homomorphism of sheaves of
$\cO_{X'}$-algebras
$$S^{\bullet}(L')\to F_{X/S, *}(Z(\Lambda_L))\subset \Lambda_L':=F_{X/S, *}\Lambda_L .$$
In particular, it makes $\Lambda'_L$ into a quasi-coherent sheaf
of $S^{\bullet}(L')$-modules. This defines  a
quasi-coherent sheaf of $\cO_{\VV (L')}$-algebras
${\ti\Lambda}_{L}'$ on the total space $\VV(L')$ of $L'$. By construction
$$\pi '_* {\ti\Lambda}_{L}'={ F}_{X/S, *}{\Lambda}_L ,$$
where $\pi' : \VV(L')\to X'$ denotes the canonical projection.
Let us recall that by \cite[Theorem 4.7]{La} 
${\ti \Lambda}_L'$ is a locally free $\cO_{\VV (L')}$-module of rank $p^{m+d}$,  where $m$ is the rank of $L$.

\medskip 

The first part of the following lemma generalizes  \cite[Lemma 2.8]{Gr}.
The second part is a generalization of \cite[Lemma 6.8]{Si2}.

\begin{Lemma}\label{equivalence2}
	We have an equivalence of categories 
	$${\QCoh \, } (X, L) \simeq {\mathrm \QCoh \, } (\VV (L'), {\ti\Lambda}_{L}')$$
	such that if  $M$ is an $L$-module and $M'$ is the corresponding $\ti \Lambda _{L}'$-module then 
	$$\pi '_* M'={ F}_{X/S, *}M.$$
	Moreover, $M$ is coherent as an $\cO_X$-module if and only if $M'$ is coherent as an $\cO_{\VV (L')}$-module
	and its support is proper over $X'$ (or equivalently the closure of the support of $M'$ does not intersect the divisor at infinity). 
\end{Lemma}

\begin{proof}
	Since both $\pi '$ and $F_{X/S}$ are affine, we have the following equivalences of categories:
	$${\mathrm \QCoh \, } (X, L) \simeq {\mathrm \QCoh \, } (X, \Lambda_ L)\simeq
	{\mathrm \QCoh \, } (X', { F}_{X/S, *}{\Lambda}_L )={\mathrm \QCoh \, } (X',\pi '_* {\ti\Lambda}_{L}')
	\simeq {\mathrm \QCoh \, } (\VV (L'), {\ti\Lambda}_{L}').$$
	The second part can be proven in the same way as Lemma \ref{equivalence}.
	Namely, $M$ is coherent if and only if ${ F}_{X/S, *}M=\pi'_*M'$ is coherent, which is equivalent to the fact that $M'$ is coherent and its support is proper over $X'$. 
\end{proof}

\medskip
We have a cartesian diagram
$$\xymatrix{
	\VV (F_X^*L) \ar[d]^{\pi}\ar[r]^{\ti F_{X/S}} &\VV (L')\ar[d]^{\pi'}\\
	X\ar[r]^{F_{X/S}}&X'\\
}$$
coming from the equality $F_{X/S}^*L'=F_X^*L$. By definition we have $\ti F_{X/S, *} {\ti\Lambda}_{L}={\ti\Lambda}_{L}'$ and equivalences of Lemmas \ref{equivalence} and \ref{equivalence2}
are compatible with each other. More precisely, by the flat base change
if $M=(E, \nabla: L\to \cEnd _{\cO_S} E)$ is an $L$-module, $\ti M$ is the corresponding $\ti \Lambda _{L}$-module and $M'$ is the corresponding $\ti \Lambda _{L}'$-module then 
$$\ti F_{X/S, *}\ti M=M'.$$
The  $\cO_{\VV (L')}$-module structure on $M'$ corresponds to the $S^{\bullet} (L')$-module structure on $E'=F_{X/S,*}E$. The corresponding map $L'\to \cEnd _{\cO_{X'}} E'$ is denoted by $\psi'(\nabla)$. If we interepret the $p$-curvature of $M$ as an $\cO_X$-linear map $F^*_{X/S}L'\otimes E\to E$, take the push-forward  by $F_{X/S}$ and use the projection formula, we get
 an $\cO_{X'}$-linear map $L'\otimes E'\to E'$ corresponding to  $\psi '(\nabla)$.
Note that if $E$ has rank $r$ then $E'$ has rank $p^dr$ as an $\cO_{X'}$-module and 
$(E', \psi'(\nabla))$ is an $L'$-module, where $L'$ is considered with trivial Lie algebroid  structure.

\section{Moduli stacks in positive characteristic}\label{moduli-section}

In this section we fix the following notation.
Let $f: X\to S$  be a flat projective morphism of noetherian schemes and let 
$\cO_X (1)$ be a relatively ample line bundle.  In the following 
stability of sheaves on the fibers of $X\to S$ is considered with respect to 
this fixed polarization. Assume also that $X/S$ is 
a family of $d$-dimensional varieties satisfying Serre's condition $(S_2)$.
Let $L$ be a smooth $\cO_S$-Lie algebroid on $X$ and let us set $\Omega_L=L^*$.

\subsection{Moduli stack of $L$-modules}

Let us fix a polynomial $P$. The moduli stack of $L$-modules $\cM^{L} (X/S, P)$ is 
defined as a lax functor from $(\Sch /S) $  to the $2$-category of groupoids, where $\cM(T)$ is the category whose objects are $T$-flat families
of $L$-modules with Hilbert polynomial $P$ on the fibres of $X_T\to T$, 
and whose morphisms are isomorphisms of quasi-coherent sheaves. $\cM^{L} (X/S,
P)$ is an Artin algebraic stack for the fppf topology on $(\Sch /S)$, which is 
locally of finite type. It contains open substacks 
$\cM^{L, \rm{tf}} (X/S, P)$ and  $\cM^{L, \rm{ss}} (X/S, P)$, 
which corresponds to families of $\cO_X$-torsion free and Gieseker semistable $L$-modules, respectively.

\subsection{Hitchin's morphism for $L$-coHiggs
	sheaves}

Let us fix some positive integer $r$ and consider the functor which to an $S$-scheme $h:
T\to S$ associates
$$\bigoplus _{i=1}^rH^0(X_T/T, {{S} }^i\Omega_{L,T}).$$
By \cite[Lemma 3.6]{La} this functor is representable by an
$S$-scheme $\VV ^L  (X/S, r)$.

Let us fix a polynomial $P$ of degree $d=\dim (X/S)$, corresponding to rank $r$ sheaves.
If we consider $L$ with a trivial $\cO_S$-Lie algebroid structure (we will denote it by $L_{\rm triv}$), then the corresponding moduli stacks are denoted by $\cM^L_{\rm Dol}(X/S, P)$, $\cM^{L, \rm{tf}}_{\rm Dol} (X/S, P)$ etc. (the Dolbeaut moduli stacks of $\Omega_L$-Higgs sheaves).
One can define Hitchin's morphism
$$H_L: \cM^{L, \rm{tf}}_{\rm Dol}(X/S, P) \to \VV ^L(X/S, r)$$
by evaluating elementary symmetric polynomials $\sigma _i$ on $E\to E\otimes \Omega_{L_T}$
corresponding to an $L_{\rm triv}$-module structure on a locally free part of $E$  (see \cite[3.5]{La}).
Alternatively, we can describe it as follows. Assume $E$ is a locally free $\cO_{X_T}$-module. 
An $L_{\rm triv}$-module structure on $E$ can be interpreted  as
a section $s: \cO_T\to \cEnd E\otimes \Omega_{L_T}$. Locally this gives a matrix with values in $\Omega _{L_T}$
and we can consider the characteristic polynomial 
$$\det (t\cdot I-s)=t^r+\sigma_1 (s) t^{r-1}+...+\sigma _r (s),$$
where $t$ is a formal variable.
To see that this makes sense one needs to use the integrability condition $s\wedge s=0$ (which is obtained from
the $L_{\rm triv}$-module structure).
These local sections glue to  $$H_L((E,s))= (\sigma_1 (s), ..., \sigma _r (s))\in \VV ^L(X/S, r) (T). $$
In general, one uses this construction on a big open subset on which $E$ is locally free and uniquely extends the sections using  Serre's condition $(S_2)$.

\subsection{Langton type properness theorem}

Let $R$ be a discrete valuation ring with maximal ideal $m$ and the quotient field $K$. Let us assume that  the residue field $k=R/m$ is  algebraically closed.

\medskip

The following  Langton's type theorem for modules  over Lie algebroids is a special case of
\cite[Theorem 5.3]{La}. 

\begin{Theorem}\label{slope-Langton}
	Let $S=\Spec R $ and let $F$ be an $R$-flat $\cO_X$-coherent $L$-module of relative
	pure dimension $n$ such that the $L_K$-module $F_K=F\otimes _RK$
	is Gieseker semistable. Then there exists an  $R$-flat $L$-submodule $E\subset
	F$ such that $E_K=F_K$ and $E_k$ is a Gieseker semistable
	$L_k$-module on $X_k$.
\end{Theorem}

\subsection{Properness of the $p$-Hitchin morphism}

Assume that $S$ has characteristic $p>0$ and $L$ is a restricted $\cO_S$-Lie algebroid.
Then the $p$-curvature gives rise to a morphism of moduli stacks
$$\begin{array}{cccc}
	\Psi_L: & \cM^L(X/S, P)&\to&  \cM^{F_X^*L}_{\rm Dol}(X/S, P)\\
	&(E, \nabla)&\to&(E,\psi (\nabla)).
\end{array}
$$
If $(E, \psi(\nabla))$ is Gieseker semistable then $(E,\nabla)$ is Gieseker  semistable. However, the opposite implication fails. For example, one can consider any semistable vector bundle $G$ on a smooth projective curve $X$ defined over a field of characteristic $p$ such that $F_X^*G$ is not semistable. Then $E=F_X^*G$ has a canonical connection $\nabla$ 
with vanishing $p$-curvature. In this case $(E, \nabla)$ is semistable but  $(E, \psi(\nabla))$ is not semistable.
This shows that the morphism $\Psi_L$ does not restrict to a morphism of moduli stacks of semistable objects. 
But since Gieseker semistable $L$-modules are torsion free, we can still consider $\Psi _L: \cM^{L, ss} (X/S, P)\to \cM^{F_X^*L, \rm{tf}}_{\rm Dol}(X/S, P)$. The composition of this morphism with Hitchin's morphism $H_{F_X^*L}: \cM^{F_X^*L, \rm{tf}}_{\rm Dol}(X/S, P) \to  \VV^{ F_X^*L}(X/S,r)$ will be called a \emph{$p$-Hitchin morphism}.

\begin{Theorem}\label{p-Hitchin-properness}
Let us fix a polynomial $P$ of degree $d=\dim (X/S)$, corresponding to rank $r$ sheaves.
The $p$-Hitchin morphism $H_{L,p}: \cM^{L, ss} (X/S, P)\to  \VV^{ F_X^*L}(X/S,r)$ is universally closed.
\end{Theorem}

\begin{proof}
	Let us consider a commutative diagram
		$$\xymatrix{
		\Spec \, K \ar[d]\ar[r]&\cM ^{L, ss}(X/S, P)\ar[d]\\
		\Spec \, R \ar@{-->}[ru]\ar[r]&\VV ^{ F_X^*L}(X/S,r).\\
	}$$
We need to show existence of the dashed arrow making the diagram commutative. Taking a base change we can assume that $S=\Spec R$. Then we need to show that for a fixed  semistable  $L_K$-module $M$ on $X_K$ there exists an $R$-flat $\cO_X$-coherent $L$-module $F$ such that $M\simeq F\otimes _RK$ and $F_k$ is Gieseker semistable.
	
	\medskip
	
\emph{Step 1.} Let us show that there exists an $R$-flat  $\cO_X$-quasicoherent $L$-module $F$ 
such that $M\simeq F\otimes _RK$.
By Lemma \ref{equivalence} there exists $\ti \Lambda _{L_K}$-module $\ti M$ on $\VV (F_{X_K}^*L_K)$ such that 
$$(\pi_{X_K})_* \ti M=M,$$
where $\pi : \VV (F_X^* L)\to X$ is the canonical projection and $\pi_{X_K}$ is its restriction to the preimage of $X_K$. 
By Lemma \ref{extension} there exists an  
$\ti \Lambda _{L}$-module $\ti M '$, coherent as an $\cO_{\VV(F_X^*L)}$-module, which extends $\ti M$ via 
an open immersion $\VV(F_{X_K}^*L _K)\hookrightarrow \VV(F_X^*L)$. Then again by  
Lemma \ref{equivalence} we get the required $L$-module $F=\pi _* \ti M'$.

\medskip

\emph{Step 2.} In this step we show that $F$ is coherent as an $\cO_X$-module. 
Let $\ti N$ be the $\cO_{\VV(F_X^*L)}$-module structure on $\ti M'$. By the results of Section \ref{modules-char-p}, $\ti N$ corresponds to the $F_X^*L$-module $\Psi _L (F)$. Let us recall that we have the total spectral scheme $\WW ^L(X/S,r)\subset
\VV(F_X^*L)\times _S \VV ^{F_X^*L}(X/S, r)$, which is finite and flat over
$X\times _S \VV ^{F_X^*L} (X/S, r)$ (see \cite[p.~521]{La}). By construction the support of $\ti N_K$ coincides set-theoretically with the spectral scheme of $\Psi _{F_{X_K}^*L_K}(M)$, so it is contained in the closed subscheme $\WW ^{F_X^* L}(X/S,r)\times _{\VV ^{F_X^*L}(X/S, r)}\Spec R $ of $ \VV(F_X^*L)$ (here we use existence of $\Spec R\to \VV ^{ F_X^*L}(X/S,r)$ making the diagram at the beginning of proof commutative). This subscheme does not intersect the divisor at infinity (when fixing an appropriate relative compactification of $\VV(F_X^*L)$) and it contains the support of $\ti N$. So the support of $\ti N$ is proper over $\Spec R$, which implies that $F$ is coherent as an $\cO_X$-module.

\medskip

\emph{Step 3.} Now the required assertion follows from Theorem \ref{slope-Langton}.
\end{proof}

\medskip

In the proof of the above theorem we used the following lemma.

\begin{Lemma}\label{extension}
Let $X$ be a quasi-compact and quasi-separated scheme and let $j: U\to X$ be 
a quasi-compact open immersion of schemes. Let $\cA$ be a sheaf of 
associative and unital (possibly non-commutative) $\cO_X$-algebras, which is locally free of finite rank as a (right) 
$\cO_X$-module. Let $E$ be a left $\cA _U$-module, which is quasi-coherent of finite type as an $\cO_U$-module.
Then there exists  a left $\cA$-module $G$, which is quasi-coherent of finite type as an $\cO_X$-module
and such that $G_U\simeq E$ as $\cA_U$-modules.
\end{Lemma}

\begin{proof}
By \cite[Tag 01PE and {Tag 01PF}]{St} there exists  a quasi-coherent $\cO_X$-submod\-ule  $E'\subset j_*E$ such that $E'|_U=E$ and $E'$ is  of finite type. Note that $j_*E$ is a $j_*(\cA_U)$-module and $j_*(\cA)$
is a $\cA$-module. So we can consider $G:=\cA\cdot E' \subset j_*E$. Clearly, $G$ is an $\cA$-module  
of finite type (as it is the image of $\cA\otimes E'$) and $G_U\simeq E$ as $\cA_U$-modules.
\end{proof}

\begin{Remark}
Theorem \ref{p-Hitchin-properness} was stated in passing in \cite[p.~531, l.~2-3]{La} but without a full proof. 
\end{Remark}

\subsection{Properness of the $p$-Hitchin morphism II}

Assume that $S$ has characteristic $p>0$ and $X/S$ is smooth of relative dimension $d$.
We assume that $L$ is a restricted $\cO_S$-Lie algebroid and we use notation from Subsection \ref{F-twist}. We also fix a polynomial $P$ of degree $d=\dim (X/S)$, corresponding to rank $r$ sheaves.
Then the $p$-curvature gives rise to a morphism of moduli stacks
$$\begin{array}{cccc}
	\Psi_L': & \cM^L(X/S, P)&\to&  \cM^{L'}_{\rm Dol}(X'/S, P')\\
	&(E, \nabla)&\to&(F_{X/S, *}E,\psi '(\nabla)),
\end{array}
$$
where $P'$ is the Hilbert polynomial of the corresponding push-forward by $F_{X/S}$.
As before we can consider 	the composition of  $\Psi _L': \cM^{L, ss} (X/S, P)\to \cM^{L',  \rm{tf}}_{\rm Dol}(X'/S, P')$ with Hitchin's morphism $H_{L'}: \cM^{L',  \rm{tf}}_{\rm Dol}(X'/S, P) \to  \VV^{ L'}(X'/S,p^d r)$. This will be called a
\emph{$p'$-Hitchin morphism}.

Essentially the same proof as that of Theorem \ref{p-Hitchin-properness} gives the following theorem: 

\begin{Theorem}\label{Hodge-Hitchin-properness2}
	Let us fix a polynomial $P$ of degree $d$, corresponding to rank $r$ sheaves.
	The $p'$-Hitchin morphism $H_{L', p'}:\cM^{L, ss} (X/S, P)\to   \VV^{ L'}(X'/S,p^d r)$ is universally closed.
\end{Theorem}

We have the following diagram
	$$\xymatrix{
\cM^{L, ss} (X/S, P)\ar[d]^{H_{L', p'}}\ar[r]^{H_{L, p}}& \VV^{ F_{X/S}^* L'}(X/S,r)	\\
 \VV^{ L'}(X'/S,p^d r) & \VV^{ L'}(X'/S,r)\ar[u]^{F_{X/S}^*}\ar[l]^{(\cdot )^{p^d}}.\\
}$$
But there are no natural maps between $\VV^{ F_{X/S}^* L'}(X/S, r)$ and
 $\VV^{ L'}(X'/S,p^d r)$, and in general the maps  $H_{L, p}$ and $H_{L', p'}$ do not factor through $\VV^{ L'}(X'/S,r)$,
 so Theorems \ref{p-Hitchin-properness} and \ref{Hodge-Hitchin-properness2} give different results. 
 However, these maps factor through $\VV^{ L'}(X'/S,r)$ in the following special case where $L=\TT_{X/S}$ is the canonical restricted Lie algebroid associated to $T_{X/S}$ with $\alpha=\id$, the usual Lie bracket for derivations and $(\cdot)^{[p]}$ given by sending $D$ to the derivation acting like the differential operator $ D^p$.
 Let us recall that for this Lie algebroid by \cite[Proposition 2.2.2 and Theorem 2.2.3]{BMR} $\ti \Lambda _L' =F_{X/S, *}\ti \Lambda_L$ is a sheaf of Azumaya $\cO_{\VV (L')}$-algebras and  we have a canonical isomorphism
 $$\varphi: F_{X/S}^{*} F_{X/S, *}\ti \Lambda_L\to \cEnd_{\cO_{\VV ( F^*_{X/S}L')}}\ti \Lambda_L$$	
 of sheaves of rings (note that our conventions of left and right modules are opposite to those in \cite{BMR}). 

 Let $\cA$ be a sheaf of $\cO_Y$-algebras. In the following ${\mathrm {QCoh} \, } _{\mathrm{fp}} (Y, \cA)$ denotes the category of left $\cA$-modules, which are quasicoherent and locally finitely presented as $\cO_Y$-modules. In the formulation of the following theorem we use the above isomorphism to identify $ F_{X/S}^{*} F_{X/S, *}\ti \Lambda_L$-module structure on $ F_{X/S}^{*} F_{X/S, *}\ti M$
 with the corresponding $\cEnd_{\cO_{\VV ( F^*_{X/S}L')}}\ti \Lambda_L$-module structure.
 
 \begin{Theorem}\label{Laszlo-Pauly}
 Let $Y\to T$ be a smooth morphism with $T$ of characteristic $p>0$ and let  $L=\TT_{Y/T}$.
 Then we have equivalences of categories
 $$\ti F_{Y/T}^{*} \ti F_{Y/T, *} : {\mathrm {QCoh} \, } _{\mathrm{fp}} (\VV (F_{Y/T}^*L'), \ti \Lambda _L )\to 
  {\mathrm {QCoh} \, } _{\mathrm{fp}} (\VV (F_{Y/T}^*L'), \cEnd_{\cO_{\VV ( F^*_{Y/T}L')}}\ti \Lambda_L ) $$
and 
 $$ \ti \Lambda_L ^*\otimes _{\cO_{\VV ( F^*_{Y/T}L')}} : {\mathrm {QCoh} \, } _{\mathrm{fp}} (\VV (F_{Y/T}^*L'), \ti \Lambda _L )\to {\mathrm {QCoh} \, } _{\mathrm{fp}} (\VV (F_{Y/T}^*L'), \cEnd_{\cO_{\VV ( F^*_{Y/T}L')}}\ti \Lambda_L ).$$
 Moreover, there exists a natural transformation
 $$\varphi: \ti F_{Y/T}^{*} \ti F_{Y/T, *} \to  \ti \Lambda_L ^*\otimes _{\cO_{\VV ( F^*_{Y/T}L')}},$$ 
 which is an isomorphism of functors.
 \end{Theorem}
 
 \begin{proof}
 	The fact that $\ti \Lambda_L ^*\otimes _{\cO_{\VV ( F^*_{Y/T}L')}}$ is an equivalence of categories follows from the standard Morita equivalence. So to prove the theorem it is sufficient to show an isomorphism of functors $\varphi$.
The natural transformation $\varphi$ is induced from the fact that $\ti F_{Y/T}^{*} $ is left adjoint to 
$\ti F_{Y/T, *}$. 
Let  $M$ be an $L$-module, which is quasi-coherent and locally finitely presented as an $\cO_Y$-module.
Let $\ti M$ be the $\ti \Lambda_L$-module corresponding to $M$. 
Then the canonical
map $\ti F_{Y/T}^{*} \ti F_{Y/T, *}\ti M\to \ti M$ induces 
$$\varphi: \ti F_{Y/T}^{*} \ti F_{Y/T, *} \ti M\to \ti \Lambda_L ^*\otimes _{\cO_{\VV ( F^*_{Y/T}L')}} \ti M$$	and we need to show that it is an isomorphism of $\cEnd_{\cO_{\VV ( F^*_{Y/T}L')}}\ti \Lambda_L$-modules.	

Note that this statement is local both in $T$ and $Y$, so we can assume that they are both affine and $\tilde M$
can be written as the cokernel of the homomorphism  $\ti \Lambda_L^{\oplus m}\to \ti \Lambda_L^{\oplus n}$
of trivial $\ti \Lambda_L$-modules of finite rank. 
This induces a commutative diagram
	$$\xymatrix{
		\ti F_{Y/T}^{*} \ti F_{Y/T, *} (\ti \Lambda_L^{\oplus m} )\ar[d]^{\simeq}\ar[r]&
		\ti F_{Y/T}^{*} \ti F_{Y/T, *} (\ti \Lambda_L^{\oplus n}) \ar[r]\ar[d]^{\simeq} 
		&\ti F_{Y/T}^{*} \ti F_{Y/T, *}\ti M\ar[r]\ar[d] & 0	\\
\ti \Lambda_L ^*\otimes _{\cO_{\VV ( F^*_{Y/T}L')}} \ti \Lambda_L^{\oplus m} \ar[r]&\ti \Lambda_L ^*\otimes _{\cO_{\VV ( F^*_{Y/T}L')}} \ti \Lambda_L^{\oplus n} \ar[r] &\ti \Lambda_L ^*\otimes _{\cO_{\VV ( F^*_{Y/T}L')}} \ti M\ar[r] & 0	,\\
	}$$
where the two vertical maps are isomorphisms by  \cite[Proposition 2.2.2]{BMR}. This implies 
that the last vertical map is also an isomorphism.
 \end{proof}

\subsection{Properness of the Hodge--Hitchin morphism}

Let us define a restricted Lie algebroid $L=\TT_{X/S, \AA^1}$ on $X\times \AA^1/S\times \AA^1$  
by setting $L:=p_1^*T_{X/S}$ with Lie bracket given by $[\cdot ,
\cdot]_{L}:=p_1^*[\cdot , \cdot]_{\TT _{X/S}}\otimes t $, the anchor map $\alpha :=p_1^*\id \otimes t$
and the $p$-th power operation given by $(\cdot)_L^{[p]}=p_1^*(\cdot)^{[p]}_{\TT _{X/S}}\otimes t^{p-1}$.

Then $\cM^{L, ss} (X/S, P)$ is the Hodge moduli stack, i.e., the moduli stack of semistable modules with $t$-connections, 
and we denote it by $\cM_{Hod} (X/S, P)$. 
 The following result follows easily from Theorem \ref{Laszlo-Pauly}. If $X$ is a smooth projective curve defined over an algebraically closed field of characteristic $p$ this result was proven in \cite[Proposition 3.2]{LP}. 

\begin{Corollary}\label{corrected-commutativity}
If  $L=\TT_{X/S, \AA^1}$ then there exists a morphism 
$$\ti H_{p}:\cM_{Hod} (X/S, P)\to  \VV^{ L'}(X'/S,r)= \VV^{ T_{X'/S}}(X'/S,r)\times \AA^1,$$
called the \emph{Hodge--Hitchin morphism}, making the diagram
	$$\xymatrix{
\cM_{Hod} (X/S, P)\ar[d]^{H_{L', p'}}\ar[r]^{H_{L, p}}\ar[rd]^{\ti H_{L, p}}& \VV^{ F_{X/S}^* L'}(X/S,r)	\\
	\VV^{ L'}(X'/S,p^d r) & \VV^{ L'}(X'/S,r)\ar[u]^{F_{X/S}^*}\ar[l]^{(\cdot )^{p^d}}\\
}$$
commutative.
\end{Corollary}

\begin{proof}
We have a cartesian diagram	
		$$\xymatrix{
	\VV^{ L'}(X'/S,r) \ar[d]^{(\cdot )^{p^d}} \ar[r]^{F_{X/S}^*}& \VV^{ F_{X/S}^* L'}(X/S,r)\ar[d]^{(\cdot )^{p^d}} \\
		\VV^{ L'}(X'/S,p^d r) \ar[r]^{F_{X/S}^*}&  	\VV^{ F_{X/S}^* L'}(X/S,p^d r),\\
	}$$
so it is sufficient to show that the diagram 
	$$\xymatrix{
	\cM_{Hod} (X/S, P)\ar[d]^{H_{L', p'}}\ar[r]^{H_{L, p}}& \VV^{ F_{X/S}^* L'}(X/S,r)\ar[d]^{(\cdot )^{p^d}} \\
	\VV^{ L'}(X'/S,p^d r) \ar[r]^{F_{X/S}^*}&  	\VV^{ F,_{X/S}^* L'}(X/S,p^d r),\\
}$$
is commutative. Note that $(H_{L, p})^{p^d}$ is given by sending an $ L_{X_T/T}$-module $M$ to the characteristic polynomial of $\ti \Lambda_L ^*\otimes _{\cO_{\VV ( F^*_{Y/T}L')}} \ti M$. Here we use the proof of \cite[Lemma 4.6]{La}, which shows that over the inverse image of an open subset $U\subset T$ on which $L_T$ is a free $\cO_{X_T}$-module,  $\ti \Lambda_{L}$ is a free $_{\cO_{\VV ( F^*_{Y/T}L')}}$-module of rank $p^d$ . Since $F_{X/S}^*\circ H_{L', p'}$ is given by sending $M$ to the characteristic polynomial of $\ti F_{X_T/T}^{*} \ti F_{X_T/T, *} \ti M$, the  corollary follows from Theorem \ref{Laszlo-Pauly}.
\end{proof}

\begin{Remark}\begin{enumerate}
		\item If $X$ is a smooth projective  variety  defined over an algebraically closed field of characteristic $p$ and one restricts to $t=1$ (the de Rham case) commutativity of the upper triangle in the diagram was mentioned in \cite[2.5]{EG} and attributed to \cite{LP}.
		However, the proof of \cite{LP} uses an assumption that $X$ is a curve. Recently, M. de Cataldo, A. F. Herrero and S. Zhang noticed that even in that case the proof of \cite{LP} needs some additional arguments.
		\item The above corollary generalizes also \cite[Theorem 2.17]{EG}. This theorem  shows existence of the lower triangle in the diagram after restricting to the de Rham and locally free part.
	\end{enumerate}
\end{Remark}

\medskip

For a curve $X$ the following corollary is one of the main theorems of \cite{dCZ}.

\begin{Corollary}\label{Hodge-Hitchin-properness}
	The Hodge--Hitchin morphism
	$$\ti H_{p}:\cM_{Hod} (X/S, P)\to \VV^{ T_{X'/S}}(X'/S,r)\times \AA^1$$
	is of finite type, universally closed and S-complete.
In particular, the Hodge moduli space of relative integrable $t$-connections on $X/S$ with fixed Hilbert polynomial $P$  is proper over $\VV^{ T_{X'/S}}(X'/S,r)\times \AA^1$.
\end{Corollary}

\begin{proof}
The fact that the moduli  stack is of finite type follows from \cite{La0}. The fact that it is universally closed follows from the definition and Theorem \ref{p-Hitchin-properness}. S-completness follows from  Theorem \ref{S-completness}. The second part of the theorem follows from the first one.	
\end{proof}

\section*{Acknowledgements}
The author would like to thank Jochen Heinloth for very useful and helpful explanations of \cite{AHLH} and Mark Andrea de Cataldo for asking questions that forced the author to give a proof of Corollary \ref{Hodge-Hitchin-properness} and for point out an inaccuracy in the first proof of Corollary \ref{corrected-commutativity}. The author would also like to thank Daniel Greb and Matei Toma for pointing out \cite{GT2}. The author was partially supported by Polish National Centre (NCN) contract numbers
2018/29/B/ST1/01232.


\footnotesize

\begin{thebibliography}{AHLH}

\bibitem {AHLH} Alper, Jarod; Halpern-Leistner, Daniel; Heinloth, Jochen  Existence of moduli spaces for algebraic stacks. {\tt arXiv:1812.01128}, preprint (2018). 


\bibitem {At} Atiyah, Michael Francis On the Krull-Schmidt theorem with application to sheaves.
\emph{Bull. Soc. Math. France} {\bf 84} (1956), 307--317.

\bibitem {BMR}
Bezrukavnikov, Roman; Mirkovi\'c, Ivan; Rumynin, Dmitriy Localization of modules
for a semisimple Lie algebra in prime characteristic. With an
appendix by Bezrukavnikov and Simon Riche, \emph{Ann. of Math.
	(2)} {\bf 167} (2008), 945--991.


\bibitem {BX} Blum, Harold; Xu, Chenyang Uniqueness of K-polystable degenerations of Fano varieties. 
\emph{Ann. of Math. (2)} {\bf 190} (2019), 609--656.



\bibitem {dCZ}
de Cataldo, Mark Andrea A.; Zhang, Siqing  Projective completion of moduli of t-connections on curves in positive and mixed characteristic. \emph{Adv. Math.} {\bf 401} (2022), Paper No. 108329, 43 pp.

\bibitem {EG}  Esnault, H\'el\`ene; Groechenig, Michael Rigid connections and $F$-isocrystals. \emph{Acta Math.} {\bf 225} (2020), 103--158.

\bibitem {Fa}  Faltings, Gerd Moduli-stacks for bundles on semistable curves. \emph{Math. Ann.} {\bf 304} (1996), 489-515.


\bibitem {FvdK}  Franjou, Vincent; van der Kallen, Wilberd 
Power reductivity over an arbitrary base. 
\emph{Doc. Math.} (2010), Extra vol.: Andrei A. Suslin sixtieth birthday, 171--195.

\bibitem {GT} Greb, Daniel; Toma, Matei 
Compact moduli spaces for slope-semistable sheaves. 
\emph{Algebr. Geom.} {\bf 4} (2017), 40--78. 

\bibitem{GT2}
Greb, Daniel; Toma, Matei Moduli spaces of sheaves that are semistable with respect to a K\"ahler polarisation. \emph{J. \'Ec. polytech. Math.} {\bf 7} (2020), 233--261.

\bibitem {Gr} Groechenig, Michael 
Moduli of flat connections in positive characteristic. 
\emph{Math. Res. Lett.} {\bf 23} (2016), 989--1047.

\bibitem {HL}  Huybrechts, Daniel; Lehn, Manfred The geometry of moduli spaces of sheaves. Second edition. \emph{Cambridge Mathematical Library.} Cambridge University Press, Cambridge, 2010. xviii+325 pp.

\bibitem {Ka}  Katz, Nicholas M. Exponential sums and differential equations. \emph{Annals of Mathematics Studies} {\bf 124}. Princeton University Press, Princeton, NJ, 1990. xii+430 pp.

\bibitem {La0} Langer, Adrian
Semistable sheaves in positive characteristic. \emph{Ann. of Math.} {\bf 159} (2004), 251--276.

\bibitem {La} Langer, Adrian Semistable modules over Lie algebroids in positive characteristic. \emph{Doc. Math.} {\bf 19} (2014), 509--540.

\bibitem {Lt}  Langton, Stacy G. Valuative criteria for families of vector bundles on algebraic varieties.
\emph{Ann. of Math. (2)} {\bf 101} (1975), 88--110.

\bibitem {LP} Y. Laszlo, Ch. Pauly, On the Hitchin morphism in positive characteristic.
\emph{Internat. Math. Res. Notices} {\bf 3} (2001), 129--143.

\bibitem {Ma}  Maruyama, Masaki Moduli spaces of stable sheaves on schemes. Restriction theorems, boundedness and the GIT construction. With the collaboration of T. Abe and M. Inaba. With a foreword by Shigeru Mukai. \emph{MSJ Memoirs} {\bf 33} Mathematical Society of Japan, Tokyo, 2016. xi+154 pp.

\bibitem {Ol} Olsson, Martin Algebraic spaces and stacks. \emph{American Mathematical Society Colloquium Publications} {\bf 62}. American Mathematical Society, Providence, RI, 2016. xi+298 pp. 

\bibitem {Se} Seshadri, C. S. Geometric reductivity over arbitrary base.
\emph{Advances in Math.} {\bf 26} (1977), 225--274. 

\bibitem {Si2} Simpson, Carlos Moduli of representations of the
fundamental group of a smooth projective variety. II, \emph{Inst.
	Hautes \'Etudes Sci. Publ. Math.} No. {\bf 80} (1994), 5--79
(1995).


\bibitem {St} The Stacks project authors,  \emph{The Stacks project}. {\tt https:/\!\!/stacks.math.columbia.edu}, 2021.

\bibitem {vdK}  van der Kallen, Wilberd Reductivity properties over an affine base. \emph{Indag. Math. (N.S.)}, Special issue to the memory of T. A. Springer, {\tt https:/\!\!/doi.org/10.1016/j.indag.2020.09.009} 

\end{thebibliography}
\end{document}